\documentclass[preprint,10pt]{elsarticle}
\usepackage{amssymb}
\usepackage{amsmath}
\usepackage{graphicx}
\usepackage{xcolor}
\usepackage{extarrows}
\usepackage{pifont}
\usepackage{enumerate}
\usepackage[shortlabels]{enumitem}
\usepackage[english]{babel}
\usepackage[latin1]{inputenc}
\usepackage[T1]{fontenc}
\usepackage[colorlinks]{hyperref}
\usepackage[margin=0.8in]{geometry}
\usepackage{amsthm}
\everymath{\displaystyle}
\usepackage{pgf,tikz}
\usepackage{mathrsfs}
\usetikzlibrary{arrows}
\AtBeginDocument{%
\hypersetup{%
linkcolor=blue,%
citecolor=blue,%
}%
}%
\journal{Elsevier}
\newcommand{\barint}{
         \rule[.036in]{.12in}{.009in}\kern-.16in
          \displaystyle\int  }
\def\R{{\mathbb{R}}}

\def\rn{{\mathbb{R}^{N}}}

\newcommand{\supp}{{\mathrm{supp}}}

\newcommand{\mconv}{\xrightarrow[k\to\infty]{mod}}

\newcommand{\mconvr}{\xrightarrow[r\to\infty]{mod}}

\def\ve{{\varepsilon}}

\newcommand{\wbM}{{\widetilde{M}}}
\newcommand{\Mb}{{\widetilde{M}_{x,\ve}}}
 
\newtheorem{conj}{\bf Conjecture}

\newcommand{\vp}{\varphi}

\newtheorem{theorem}{Theorem} 
\newtheorem{definition}{Definition}

\newtheorem{lemma}{Lemma}
\newtheorem{corollary}{Corollary}

\newtheorem{remark}{Remark}

\DeclareMathOperator{\dist}{dist}

\begin{document}
\begin{frontmatter}
\title{Gossez's approximation theorems in the Musielak-Orlicz-Sobolev spaces}  

\author[label1]{Youssef Ahmida}
\author[label2]{Iwona Chlebicka}
\author[label2]{Piotr Gwiazda}
\author[label1]{Ahmed Youssfi}
\address[label1]{National School of Applied Sciences,\\ 
Laboratory of Engineering, Systems and Applications (LISA),\\
Sidi Mohamed Ben Abdellah University\\
My Abdellah Avenue, Imouzzer road, P.O. Box 72 F\`es-Principale, 30 000 Fez, Morocco}
\address[label2]{Institute of  Mathematics, Polish Academy of Sciences,\\ ul. \'{S}niadeckich 8, 00-656 Warsaw, Poland\\
{\textrm{e-mail addresses: \texttt{ youssef.ahmida@usmba.ac.ma, p.gwiazda@mimuw.edu.pl, i.skrzypczak@mimuw.edu.pl, ahmed.youssfi@gmail.com, ahmed.youssfi@usmba.ac.ma}}}}
\begin{abstract}
We prove the density  of smooth functions  in the modular topology in the Musielak-Orlicz-Sobolev spaces essentially extending the results of Gossez~\cite{GJP2} obtained in the Orlicz-Sobolev setting. We impose new systematic regularity assumption on $M$ which allows to study the problem of density unifying and improving the known results in the Orlicz-Sobolev spaces, as well as the variable exponent Sobolev spaces. 

We confirm the precision of the method by showing the lack of the Lavrentiev phenomenon in the double-phase case. Indeed, we get the modular approximation of $W^{1,p}_0(\Omega)$ functions by smooth functions in the double-phase space governed by the modular function $H(x,s)=s^p+a(x)s^q$ with $a\in C^{0,\alpha}(\Omega)$ excluding the Lavrentiev phenomenon within the sharp range $q/p\leq 1+\alpha/N$.   See~\cite[Theorem~4.1]{min-double-reg1} for the sharpness of the result.
\end{abstract}
\end{frontmatter}
\section{Introduction  and statement of the results}


The highly challenging and important part of the analysis in the general Musielak-Orlicz spaces is giving a~relevant structural condition implying approximation properties of the space. In general, smooth functions are not dense in norm in~this type of spaces. 
In the seminal paper~\cite{GJP2} Gossez proves that weak derivatives in the Orlicz-Sobolev spaces are strong derivatives with respect to the modular topology. We extend the idea to the Musielak-Orlicz-Sobolev setting, where the modular function depends also on the spacial variable, i.e. $M=M(x, s )$.

It is known that approximating properties of the Musielak-Orlicz-Sobolev spaces depends on the regularity of~$M$ (see below information on the Lavrentiev phenomenon). Our goal is to find optimal assumptions on interplay between types of regularity with respect to each of the variables that ensures the density of smooth functions. We relax the condition typical in the context of density in the variable exponent spaces, namely the log-H\"older continuity of the exponent. Moreover, we exclude the Lavrentiev gap in the double phase spaces within the sharp range of parameters.

We give density results provided that $M(x, s )$ is convex in $ s $  and $\varphi$-regular (i.e. $M(x, s )\leq\varphi(|x-y|, s )M(y, s )$ under certain type of regularity condition on $\vp$). See conditions (\ref{incBM} and \ref{X1}) or (\eqref{M>p}, \ref{incBM}, and \ref{X1p}) for details. Obviously, $\vp$--regularity is skipped in the Orlicz setting (when $M=M( s )$). 

\subsection*{The Musielak-Orlicz spaces}
In the monograph of Musielak \cite{MJ}, the series of papers written by Hudzik~\cite{HH3,HH4,HH2,HH1}, and the papers by~Skaff~\cite{Sk1,Sk2}  basic background of the Musielak-Orlicz spaces $L_M$ and the Musielak-Orlicz-Sobolev spaces $W^{m}L_M$ built upon a~$\Phi$-function $M(x,s)$ (Definition~\ref{def:Phifn})  was developed. This functional framework has received significant attention for the last two decades both --- from the~theoretical and from the~applied point of view. In particular,  deep study of~the~variable exponent Lebesgue and the Sobolev spaces (i.e. when $M(x,s)=s^{p(x)}$) was conducted, see \cite{DHHR,KR} and the references therein. When $M(x,s)$ is independent of the first argument, we are led to the classical framework of~the~Orlicz and the Orlicz-Sobolev spaces exposed in \cite{AF,GJP2,KJF,rao-ren}. More and more attention is paid to investigation on the double-phase space, e.g.~\cite{min-double-reg1,colombo}.

The typical applications of the spaces include models of electrorheological fluids~\cite{mingione02,raj-ru1,el-rh2}, image restoration processing~\cite{chen}, non--Newtonian fluid dynamics~\cite{gwiazda2}, Poisson equation~\cite{DHHR}, elasticity equations~\cite{Filip,gwiazda3,zhikov97}, and thermistor model~\cite{zhikov9798}. Problems in various types of the Musielak-Orlicz spaces are widely considered from analytical point as well, inter alia the highly modern calculus of variations deals with minimization of the variational integrals~\cite{bcm17,min-double-reg1,colombo,ZV} 
\begin{equation*}
\begin{split}&\min_{u}\int_{\Omega}|\nabla u|^{p(x)}dx, \quad  \min_{u}\int_{\Omega}|\nabla u|^{p(x)}\log (e+|\nabla u|)dx, \quad  \min_{u}\int_{\Omega}|\nabla u|^{p}+a(x)|\nabla u|^{q}dx,\\ &\min_{u}\int_{\Omega} b(x, w)(|\nabla u|^{p}+a(x)|\nabla u|^{q})dx\quad \text{ and }\quad
\min_{u}\int_{\Omega}(e^{p(x)|\nabla u|}-1)dx. \end{split}
\end{equation*}
 See also how the problem of minimisation is examined in the Musielak-Orlicz setting under $\Delta_2/\nabla_2$-conditions~\cite{hht}.

\subsection*{The Lavrentiev phenomenon}
The Musielak-Orlicz spaces do not inherit all the good properties of the classical Sobolev spaces. Besides reflexivity and separability, which are not - in general - the properties we deal with, the problems with density can also appear and is related to the so-called Lavrentiev phenomenon. We meet it when the  infimum of the~variational problem over the~smooth functions is strictly greater than infimum taken over the set of all functions satisfying the~same boundary conditions, cf.~\cite{LM,ZV}. The notion of the Lavrentiev phenomenon became naturally generalised to~describe the situation, where functions from certain spaces cannot be approximated by regular ones (e.g. smooth).

It is known that in the case of the variable exponent spaces, the Lagrangian $M(x,s)=|s|^{p(x)}$ can exhibit the Lavrentiev phenomenon if $p(\cdot)$ is not regular enough (see e.g.~\cite[Example~3.2]{ZV}). The canonical, but not optimal, assumption ensuring density of smooth function in norm topology in the variable exponent spaces is the log-H\"older continuity of the exponent $p(\cdot)$. The double-phase spaces (with $M(x,s)=|s|^{p}+a(x)|s|^{q}$) can also support the Lavrentiev phenomenon~\cite{min-double-reg2,min-double-reg1,ELM}, where the authors provide sharp result.

The mentioned results show that the strong closure of the smooth functions can be not relevant type of useful approximation in the spaces with a not sufficiently regular modular function. We provide here sufficient conditions to avoid the Lavrentiev phenomenon.  Let us point out that this type of result can be used in order to get e.g.~regularity of minimisers cf.~\cite{min-double-reg1}.

\subsection*{Approximation results in the Musielak-Orlicz spaces}

\par An earlier density result of smooth functions in the Musielak-Orlicz-Sobolev spaces $W^mL_M(\mathbb{R}^N)$ with respect to the strong (norm) topology was proved first by Hudzik \cite[Theorem 1]{HH1} assuming the $\Delta_2$-condition (Definition~\ref{delta2def}) on the $\Phi$-function $M$  and the extra hypothesis
$$
\int M(x,|u_\varepsilon|)dx \leq c\int M(x,|u|)dx,
$$
where $u_\varepsilon$ is the mollification defined in Section \ref{section2}. Recently in a bounded domain $\Omega$, the same result (without the extra hypothesis) was proved in~\cite[Theorem 6.5]{HHK} using the boundedness of the maximal operator.

Density of smooth functions in the Musielak-Orlicz-Sobolev spaces with respect to the modular topology (under the log-H\"older-continuity-type restriction on the modular function) was claimed for the first time in \cite{BDV} in $\Omega=\mathbb{R}^N$ and then for a bounded star-shape Lipschitz domain $\Omega$ in \cite{BV}. Nonetheless, the proof involved an essential gap. The Jensen inequality was used for the infimum of convex functions, which obviously is not necessarily convex. The proof was fixed, under slightly changed assumptions, in~the~elliptic case in~\cite[Theorem~2.2]{pgisazg1} and in the parabolic case in \cite[Theorem~2.1]{pgisazg2}.

In \cite{BDV,BV}  function $M$ is assumed to satisfy the $\log$-H\"older continuity condition, i.e. that there exists a~constant $A>0$ such that for all $s\geq 1$,
\[
\frac{M(x,s)}{M(y,s)}\leq  s^{-\frac{A}{\log|x-y|}} \qquad \forall x,y\in \Omega \mbox{ with } |x-y|\leq\frac{1}{2}.
\]
Note that imposing this assumption makes sense only for big values of $s$, since due to $x$/$y$ symmetry forces the~fraction on the left-hand side above has to be estimated from above by the quantity bigger or equal to $1$. 

In the isotropic case (when $M=M(x,|\xi|)$) in~\cite{pgisazg1}, for smooth approximation in \[\{u\in W^{1,1}_0(\Omega):\ Du\in L_M(\Omega,\R^N)\},\] it suffices to impose on an $N$-function $M$ continuity condition of log-H\"older-type with respect to $x$, namely for each $s\geq 0$ and $x,y $ such that $|x-y|< {1}/{2}$ we have
\begin{equation}
\label{M2'} \frac{M(x,s)}{M(y,s)}\leq\max\left\{  s ^{-\frac{a_1}{\log|x-y|}}, b_1^{-\frac{a_1}{\log|x-y|}}\right\} \qquad \text{with some}\ a_1>0,\,b_1\geq 1.
\end{equation} 
In the variable exponent case the above condition relates to standard log-H\"older continuity of the exponent. Note that the results in \cite{BDV,BV,ASG} do not cover the  functions $e^{sp(x)}-1$ and $t^{p(x)}/{p(x)}$ and the two phase function $s^{p}+a(x)s^{q}$ unless $p$ and $a$ are constant functions, which are not excluded in our setting. 

\subsection*{The framework}
Our aim in this paper is to provide new systematic conditions that guarantee the density of smooth functions in Musielak-Orlicz-Sobolev spaces $W^{m}L_M(\Omega)$ upon a wide class of $N$-functions $M(x,s)$.

\begin{definition}[$\phi$-function, $\Phi$-function, $N$-function]\label{def:Phifn} A real function $M$ : $\Omega\times\mathbb{R}^{+}\to\mathbb{R}^{+}$ is called a $\phi$-function, written $M\in\phi$, if  $x\mapsto M(x,s)$ is a~measurable function for all $s\geq 0$, 
$s\mapsto M(x,s)$ is a convex function for a.e. $x\in\Omega$ with $M(x,0)=0$,  $\mathrm{ess\,inf}_{x\in\Omega}M(x,1)>0$, $M(x,s)\rightarrow+\infty \mbox{ as }s\rightarrow+\infty$.
\par A $\phi$-function is called $\Phi$-function, written $M\in\Phi$, if furthermore for a.e. $x\in\Omega$
\begin{equation*}\label{eq10.1}
\lim_{s\to 0} \frac{M(x,s)}{s}=0\quad \mbox{ and }\quad \lim_{s\to \infty} \frac{M(x,s)}{s}=\infty.
\end{equation*} 

A $\Phi$-function is called an $N$-function, if it is strictly increasing with respect to the second variable.
\end{definition}

There are examples of $\Phi$-functions listed below
\begin{equation*}
\begin{array}{lll}
M_1(x,s)= |s|^{p(x)}, \;1<p(\cdot)<\infty, &&  M_2(x,s)=|s|^{p(x)}\log(e+|s|), \; 1<p(\cdot)<\infty,\\
M_3(x,s)= \frac{1}{p(x)}[(1+|s|^2)^{\frac{p(x)}{2}} -1], \;1<p(\cdot)<\infty,&& M_4(x,s)= |s|^{p} +a(x)|s|^{q}, \;1< p< q, 0\leq a(\cdot)\in L^{1}_{loc}(\Omega),\\
M_5(x,s)= e^{|s|^{p(x)}} -1 ,\;1<p(\cdot)<\infty, &&
M_6(x,s)= \infty\chi_{(1,\infty)}(s).
\end{array}
\end{equation*}
Particular attention is paid to the space equipped with the modular function $M_1$ related to the variable exponent setting. The $\Phi$-function $M_2$ arises in plasticity when $p(\cdot)$ is a constant function. Observe that $M_4\in\Delta_2$ and if~$p^+:=\mathrm{ess\,sup}_{x\in\Omega}p(x)<+\infty$, the $\Phi$-functions $M_i$, $1\leq i\leq3$, satisfy the $\Delta_2$-condition as well. It is no longer the case of $M_5$ and $M_6$. The anisotropic and quickly-growing $\Phi$-functions $M$ which does not support reflexive spaces find an application in thermo-visco-elasticity~\cite{Filip}.

\medskip

For $M\in\phi$, the Musielak-Orlicz space $L_M(\Omega)$ (resp. $E_M(\Omega)$) is defined as the set of all measurable functions $u:\Omega\rightarrow \mathbb{R}$ such that $\int_\Omega M(x,|u(x)|/\lambda)dx<+\infty$ for some $\lambda>0$ ($\mbox{resp. for all } \lambda>0$). Equipped with the~Luxemburg norm
\begin{equation*}
\|u\|_{L_M(\Omega)}=\inf\bigg\{\lambda>0: \int_\Omega M\bigg(x,\frac{|u(x)|}{\lambda}\bigg)dx\leq 1\bigg\}.
\end{equation*}
Then $L_M(\Omega)$ is a Banach space \cite[Theorem 7.7]{MJ} and $E_M(\Omega)$ is its closed subset.  Moreover, $E_M(\Omega)$ coincides with the norm closure of the set of bounded functions in $L_M(\Omega)$, provided that for any constant $c>0$ we have $\int_\Omega M(x,c)dx<\infty$.

If $M\in\Delta_2$ (Definition~\ref{delta2def}), then  $E_M(\Omega)=L_M(\Omega)$. If the $\Phi$-function  $M$ is independent of $x$ then condition \eqref{delta2} is equivalent to the condition $M(2u)\leq kM(u)$ for all $u\geq 0$ and $k>0$ if $|\Omega|=\infty$, and to $M(2u)\leq kM(u)$ for all $u\geq u_0$ with some $u_0>0$, $k>0$ if $|\Omega|<\infty$ (see \cite[Remark 1.6]{AY}).
Unlike the Orlicz framework, the equality $E_{M}(\Omega)=L_{M}(\Omega)$ \textbf{does not imply} that the $\Phi$-function $M$ satisfies $M(x,2s)\leq kM(x,s)$ for all $s\geq 0$ and almost every $x\in\Omega$ (see \cite[Example 4.2]{HHK}).

For a positive integer $m$, we define the Musielak-Orlicz-Sobolev spaces  $W^{m}L_M(\Omega)$ and $W^{m}E_M(\Omega)$ as follows
\begin{equation*} 
W^{m}L_M(\Omega)=\big\{u\in L_M(\Omega): D^\alpha u\in L_M(\Omega), |\alpha|\leq m\big\},
\end{equation*}
\begin{equation*} 
W^{m}E_M(\Omega)=\big\{u\in E_M(\Omega): D^\alpha u \in E_M(\Omega), |\alpha|\leq m\big\},
\end{equation*}
where $\alpha=(\alpha_1,\alpha_2,\cdots,\alpha_N)$, $|\alpha|=|\alpha_1|+|\alpha_2|+\cdots+|\alpha_N|$ and $D^\alpha$ denote the distributional derivatives.
The spaces  $W^{m}L_M(\Omega)$ and $W^{m}E_M(\Omega)$ are  
endowed with the Luxemburg norm
\begin{equation}\label{eq8}
\|u\|_{W^{m}L_M(\Omega)}=\inf\bigg\{\lambda>0 : \sum_{ |\alpha|\leq m}  \int_\Omega M\bigg(x,\frac{|D^\alpha u|}{\lambda}\bigg)dx\leq 1\bigg\}.
\end{equation}
Due to~\cite{HH3}, if $M\in\Phi$, then $\big(W^{m}L_M(\Omega), \|u\|_{W^{m}L_M(\Omega)}\big)$ is a Banach space.
\par Denote by $d=d(m,N)$ the number of multi-indices $\alpha=(\alpha_1,\alpha_2,\cdots,\alpha_N)$ satisfying $|\alpha|\leq m$, that is $d=\sum_{|\alpha|\leq m}1$. Let $M\in \Phi$, define $\Pi L_M(\Omega)=\underset{i=1}{\overset{d}\Pi} L_M(\Omega)$ as the $d$-tuple cartesian product of $L_M(\Omega)$. The mapping $P:W^mL_M(\Omega)\rightarrow \Pi L_M(\Omega)$  defined for all $u\in W^mL_M(\Omega)$ by $
P(u)=(D^\alpha u)_{|\alpha|\leq m}$ 
establishes an isometric isomorphism from $W^mL_M(\Omega)$ onto the closed subspace $X=P(W^mL_M(\Omega))$ of $\Pi L_M(\Omega)$ for $m>0$.
Thus, we can  identify the space $W^{m}E_M(\Omega)$ with the closed proper subspace $P(W^mE_M(\Omega))$ of the product $\Pi E_M(\Omega)$ for $m>0$.

\par We denote by $W^{m}_0E_M(\Omega)$ the norm closure of  $\mathcal{C}^{\infty}_{0}(\Omega)$
in $W^mE_M(\Omega)$. Let $M$ be an $N$-function whose complementary $N$-function $M^{\ast}$ satisfy the condition \ref{incBM}. Then from \cite[Theorem 1.4]{AY}, the dual space of $E_{M^{\ast}}$ is isomorphic to $L_{M}$ and the following weak-$^{\ast}$ topology
$\sigma(\Pi L_M, \Pi E_{M^{\ast}})$ is well defined, thereby we define the space \[W^{m}_0L_M(\Omega)=\overline{\mathcal{C}^{\infty}_{0}(\Omega)}^{\sigma(\Pi L_M, \Pi E_{M^{\ast}})}\] i.e. as the $\sigma(\Pi L_M, \Pi E_{M^{\ast}})$
closure of $\mathcal{C}^{\infty}_{0}(\Omega)$ in $W^{m}L_M(\Omega)$. In the sequel, by $\sigma(\Pi L_M, \Pi L_{M^{\ast}})$ we denote the weak topology in $W^{m}L_M(\Omega)$.

 If an $N$-function $M$ and its conjugate $M^\ast$ satisfy both the $\Delta_2$-condition and \ref{incBM}, we have 
 \[E_{M^\ast}  \xlongequal[]{M^\ast\in\Delta_2}L_{M^\ast}\xlongequal[ ]{M \text{ satisfies } (\mathcal{M}{1})  }(E_{M})^\ast\xlongequal[]{M \in\Delta_2} (L_M)^\ast\]
and so the topologies
  \[\sigma(L_M,E_{M^*})=
\sigma(L_M,L_{M^*})
= \sigma(L_M,(E_{M})^{\ast})
=\sigma(L_M,(L_{M})^{\ast})\]
 agrees then.   Then Mazur's lemma implies that $\mathcal{C}^{\infty}_{0}(\Omega)$ is dense in $W_0^{m}L_M(\Omega)$ in every of the mentioned topologies. However, lack of the growth control of the modular function can lead out of the class with smooth approximation. We provide the~results with no growth conditions and further, to confirm precision of the method, having at least power-type growth.

\subsubsection*{Assumptions in the case without growth conditions}

\par In the sequel, we consider the following fundamental assumptions.
\begin{enumerate}[($\mathcal{M}$1)]
\item The $\Phi$-function $M$ is locally integrable, that is for any constant number $c>0$ and for any compact set $K\subset\Omega$ we have
\begin{equation*}\label{incBM}
\int_K M(x,c)dx<\infty.
\end{equation*}
\item \label{X1} There exists a function  $\varphi:\big[0, {1}/{2} ]\times\mathbb{R}^+\to\mathbb{R}^+$ such that $\varphi(\cdot,s)$ and $\varphi(x,\cdot)$ are nondecreasing functions and for all $x,y\in\overline{\Omega}$ with $|x-y|\leq\frac{1}{2}$ and for any constant $c>0$
\begin{equation*}
M(x, s )\leq\varphi(|x-y|, s )M(y, s )\qquad\text{ with } \qquad\limsup_{\varepsilon\rightarrow0^+}
\varphi(\varepsilon, c\varepsilon^{-N})<\infty.
\end{equation*} 
\end{enumerate}
\par The condition \ref{incBM} ensures that the set of bounded functions compactly supported in $\Omega$ belong to $E_M(\Omega)$ and so is the class of smooth functions compactly supported in $\Omega$. In the framework of the Orlicz spaces  \ref{incBM} is naturally verified, while for the variable Lebesgue spaces generated by the $\Phi$-function $s^{p(\cdot)}$, it is satisfied provided that $p^+=\mathrm{ess\,sup}_{x\in\Omega}p(x)<+\infty$. Incidentally, the functions essentially bounded do not belong to $E_M(\Omega)$ even if~\ref{incBM} is satisfied. Note  that if $x\mapsto M(x,s)$ is a
continuous function on $\overline{\Omega}$, then so is the complementary $\Phi$-function $M^{\ast}$ to $M$. Thus \ref{incBM} holds for $M$ if and only if it holds for $M^{\ast}$.

\par The assumption \ref{X1} that we introduce here is more general than \eqref{M2'}. The said regularity is widely connected to the question of density of smooth functions in the  Musielak-Sobolev spaces. Observe in particular that $\varphi(\tau,s)\geq 1$ for all $(\tau,s)\in[0,{1}/{2}]\times\mathbb{R}^+$. In general the $\Phi$-function $M$ \textbf{is not continuous} with respect to its first argument. Actually, only if for all $s\geq 0$ we have $\limsup_{\varepsilon\to 0^+}\varphi(\varepsilon,s)=1$, then $x\mapsto M(x,s)$ is a continuous function on $\overline{\Omega}$.
 
\medskip
We have the following examples of pairs $M$ and $\varphi$ satisfying both \ref{incBM} and \ref{X1}, and thus are admissible in our results on density of smooth functions. Computations are provided in Appendix.
\subsubsection*{Examples}
\begin{enumerate} 
\item If the $\Phi$-function $M(x,s)=M(s)$ is independent of $x$, then it satisfies obviously the \ref{X1} condition by~choosing $\varphi(\tau,s)=1$.

\item Suppose that $\Phi$-function $M(x,s)=|s|^{p(x)}$ satisfies the \ref{X1} condition with
\begin{equation*}
\varphi(\tau,s)= \max\left\{s^{\sigma(\tau)},s^{-\sigma(\tau)}\right\}.
\end{equation*}
where $\sigma : (0, {1}/{2}]\to\mathbb{R}^+$ with $\limsup_{\varepsilon\to 0}\sigma(\varepsilon)=0$ is the modulus of continuity of $p$.

Such a choice implies \cite[(2.5)]{ZV} and recovers the standard log-H\"older condition if we consider the particular case $\sigma(\tau)=- {c}/{\log\tau}$, with $0<\tau\leq {1}/{2}$. Nonethess, we can choose various $\varphi$s.  

\item Consider $M(x,s) =s^{p}+a(x)s^{q}$, where $1\leq p<q$ and nonnegative $a\in C^{0,\alpha}_{loc}(\Omega)$ with $\alpha\in(0,1]$ (i.e. $|a(x)-a(y)|\leq C_a|x-y|^\alpha$ locally), then $M$ satisfies the \ref{X1} condition with
\begin{equation*}
\varphi(\tau, s )= C_a\tau^\alpha|s|^{q-p}+1.
\end{equation*}
The assumption $\limsup_{\varepsilon\rightarrow 0^+}  \varphi(\varepsilon,c \varepsilon^{-N})<\infty$ forces $q\leq p+{\alpha}/{N}$. Below we show how to extend the range.
\item Let $M(x,s)=\frac{1}{p(x)}|s|^{p(x)}$. If $1\leq p^-=\mathrm{ess\,inf}_{x\in\Omega}p(x)\leq p(\cdot)\leq p^+=\mathrm{ess\,sup}_{x\in\Omega}p(x)<+\infty$, then we can take with
\begin{equation*}
\varphi(\tau,s)= \frac{p^+}{p^-}\max\left\{s^{-\frac{c}{\log\tau}},s^{\frac{c}{\log\tau}}\right\}.
\end{equation*}

\item  Let $M(x,s)=\sum_{i=1}^k k_i(x)M_i(s)+M_0(x,s)$, where for every $i=1,\cdots,k$ and function $k_i : \overline{\Omega}\to (0,+\infty)$ 
 there exists a nondecreasing function $\varphi_i:\big[0,{1}/{2} ]\to\mathbb{R}^+$ satisfying $
k_i(x)\leq\varphi_i(|x-y|)k_i(y)$ with $\limsup_{\ve\to 0^+}\varphi_i(\ve )<\infty,$ whereas $M_0(x,s)$ satisfies \ref{X1} with $\vp_0$. Then, we can take
\[
\vp(\tau, s)=\sum_{i=1}^k \vp_i(\tau)+\vp_0(\tau,s).
\]
\end{enumerate}

\subsubsection*{The sharp result under a growth condition}
Although it is common to make a big effort to relax growth conditions as much as possible,  our method turn out to lead to the sharp result when the modular function has at least power-type growth.  Since the approximation follows from convolution arguments, we get better regularity in $L^p$, $p>1$ than in $L^1$. In the fully general case we cannot improve~\ref{X1}, because we know only $L_M\subset L^1$. Nonetheless, when the modular function has at least power-type growth, we relax \ref{X1} and include whole the good double-phase range where the Lavrentiev phenomenon is absent (according to~\cite[Theorem~3]{ELM} or \cite[Theorem~4.1]{min-double-reg1}). 

Namely, if
\begin{equation}
\label{M>p}
M(x, s )\geq c|s|^p \qquad\text{with }\  p>1\text{ and }\  c>0,
\end{equation}
we obtain smooth approximation of $u\in W_0^{m,p}(\Omega)\cap W^{m}L_M(\Omega)$ (also of $u\in W_0^{1,p}(\Omega),$ such that $Du\in L_M(\Omega)$) provided that we assume
\begin{enumerate}[($\mathcal{M}2p$)]
\item \label{X1p} There exists a function  $\varphi:\big[0, {1}/{2} ]\times\mathbb{R}^+\to\mathbb{R}^+$ such that $\varphi(\cdot,s)$ and $\varphi(x,\cdot)$ are nondecreasing functions and for all $x,y\in\overline{\Omega}$ with $|x-y|\leq\frac{1}{2}$ and for any constant $c>0$
\begin{equation*}
M(x, s )\leq\varphi(|x-y|, s )M(y, s )\quad\mbox{ with }\quad \limsup_{\varepsilon\rightarrow0^+}
\varphi(\varepsilon,c\varepsilon^{-\frac{N}{p}})<\infty.
\end{equation*}
\end{enumerate}

The celebrated case of the Musielak-Orlicz space, when the modular function has at least power-type growth is the double-phase space. Following e.g.~\cite{bcm17,min-double-reg2,min-double-reg1,colombo}, we shall consider \[ H(x,s)=|s|^p+a(x)|s|^q\qquad\text{ with }\qquad a\in C^{0,\alpha}(\Omega),\]
and a Carath\'eodory function $F:\Omega\times\R\times\rn\to\R$ such that for any $x\in\Omega,$ $v\in\R$, and $z\in\rn$ it holds that\[\nu H(x,z)\leq F(x,v,z)\leq L H(x,z) \qquad \text{with certain}\quad 0<\nu\leq L.\] Let us denote the associated space $W^{1,H}(\Omega)=\{u\in W_0^{1,1}(\Omega): H(\cdot,D u)\in L^1(\Omega)\}$ and class $W^{1,F}(\Omega)=\{u\in W_0^{1,1}(\Omega): F(\cdot,u,D u)\in L^1(\Omega)\}$. Then we have the following sharp result.
\begin{remark}\label{rem:double-phase}
Suppose $\Omega\subset\rn$ has a segment property, $N\geq 1$, $p,q>1$, $\alpha\in(0,1)$, and nonnegative $a\in C^{0,\alpha}(\Omega)$, where
\begin{equation}\label{sharp-range} \frac{q}{p}\leq 1+\frac{\alpha}{N}.\end{equation} 

 Then for any $u\in W^{1,p}_0(\Omega)\cap W^{1,F}(\Omega)$ compactly supported in $\Omega$ there exist a sequence $\{u_k\}_k\subset C_0^\infty(\Omega)$ converging to $u$: $u_k\xrightarrow[k\to \infty]{}u$ in $W^{1,p}(\Omega)$ and $D u_k\xrightarrow[k\to \infty]{M} D u$ modularly in $W^{1,H}(\Omega)$, which entails $F(\cdot,u_k,Du_k)\xrightarrow[k\to \infty]{}F(\cdot,u,Du)$ in $L^{1}(\Omega).$
\end{remark} It results from Corollary~\ref{coro:d-p}. Indeed, when we take into account the double phase modular function $M(x,s)=H(x,s)$, then of course~\eqref{M>p} and~\ref{incBM} are satisfied. Further we need to ensure~\ref{X1p}. For $0\leq a\in C^{0,\alpha}(\Omega)$, a good choice is
\[\varphi(\tau,s)=1+c_a\tau^{\alpha}|s|^{q-p}\]
(see Appendix for calculations) and therefore \ref{X1p} is satisfied whenever the parameters satisfy~\eqref{sharp-range}. See~\cite[Theorem~4.1]{min-double-reg1}  for sharpness.


\subsection*{The results}

To formulate our density results precisely, we need to distinguish two types of  topology. We say that $\{u_k\}_k$ converges to $u$ in norm in $L_M(\Omega)$, if
$\|u_k-u\|_{L_M(\Omega)}\rightarrow 0$ as $k\rightarrow \infty$. The notion of the modular convergence is specified in the following definition.

\begin{definition}[Modular convergence]\label{def:mod-conv} A sequence $\{\xi_k\}_k$ is said to converge modularly to $\xi$ in $L_M(\Omega)$ ($\xi_k\xrightarrow[k\to \infty]{M}\xi$), if there exists $\lambda>0$ such that $\rho_{M}((\xi_k-\xi)/\lambda):=\int_{\Omega}M\left(x,{|\xi_k-\xi|}/{\lambda}\right)\, dx\rightarrow 0$ as $k\rightarrow \infty$,
 equivalently\\ if there exists $\lambda>0$ such that $\left\{M\left(x, {\xi_k}/{\lambda}\right)\right\}_k \ \text{is uniformly integrable in } L^1(\Omega)$ {and} $ \xi_k\xrightarrow[]{k\to\infty}\xi$ {in measure}.
\end{definition}

We write 
\begin{equation*}
u_k\mconv u\ \text{ in }\ W^mL_M(\Omega)\qquad\text{when}\qquad\exists_{\lambda>0}\ \forall_{|\alpha|\leq m}\quad  D^{\alpha}u_k\xrightarrow[k\to\infty]{M} D^{\alpha}u,
\end{equation*}
that is when
\begin{equation*}
\qquad\exists_{\lambda>0}\ \forall_{|\alpha|\leq m}\quad \int_{\Omega}M\Big(x,\frac{|D^{\alpha}u_k-D^{\alpha}u|}{\lambda}\Big)dx\xrightarrow[k\to\infty]{} 0.
\end{equation*}

We give below the important observation, that norm convergence in $E_M$ result from the modular one in $L_M$.
\begin{lemma}[Theorem~5.5, \cite{MJ}]\label{lem:modular-norm}
Let $M$ be an $\Phi$-function and  $u_k\xrightarrow[k\to\infty]{M} u$ in $L_M(\Omega)$ with every $\lambda>0$ (cf. Definition~\ref{def:mod-conv}), then $u_k\xrightarrow[k\to\infty]{} u$ in the norm topology in $L_M(\Omega)$. 
\end{lemma}

\medskip

In the Orlicz spaces setting, the density of $\mathcal{C}^{\infty}_{0}(\mathbb{R}^N)$ in $W^{m}E_M(\mathbb{R}^N)$ was proved by Donaldson and Trudinger in~\cite[Theorem 2.1]{DT}, while in the case of the variable exponent Sobolev spaces, the corresponding result  was proved in \cite[Theorem 3]{SSG} for bounded exponent satisfying the log-H\"older condition. We provide the following result.

\begin{theorem}\label{densityrn}
Assume that  an $N$-function $M$ satisfies \ref{incBM} and \ref{X1} (resp. \eqref{M>p}, \ref{incBM}, and \ref{X1p}). Then
\begin{itemize}
\item[1) ] $\mathcal{C}^{\infty}_{0}(\mathbb{R}^N)$ is dense in $W^{m}E_M(\mathbb{R}^N)$ with respect to the norm topology in $W^{m}L_M(\mathbb{R}^N)$.
\item[2) ] For every $u\in W^{m}L_M(\mathbb{R}^N)$, there exist $\lambda>0$ and a sequence of functions $u_k\in\mathcal{C}^{\infty}_{0}(\mathbb{R}^N)$ such that  $u_k\mconv u$ in $W^mL_M(\mathbb{R}^N)$.
\end{itemize}
\end{theorem}

\begin{remark}
If in addition $M\in \Delta_2$, then $\mathcal{C}^{\infty}_{0}(\mathbb{R}^N)$ is dense in norm topology in $W^mL_M(\mathbb{R}^N)$, because in this case $W^mL_M(\mathbb{R}^N)=W^mE_M(\mathbb{R}^N)$.
\end{remark}

We give the density result on the sets satisfying the  {segment property}.

\begin{definition}[Segment property]\label{segmpropdef}A domain $\Omega$ is said to satisfy the  {segment property}, if there exist a finite open covering $\{\theta\}_{i=1}^k$ of $\overline{\Omega}$ and a corresponding nonzero vectors $z_i\in \mathbb{R}^N$ such that  $(\overline{\Omega}\cap\theta_{i})+tz_{i}\subset\Omega$ for all $t\in (0,1)$ and $i=1,\dots,k$.\end{definition} 

This condition holds, for example, if $\Omega$ is bounded Lipschitz (see \cite{AF}). By convention we assume that the empty set satisfies the segment property. 

Let $\mathcal{C}^{\infty}_{0}(\overline{\Omega})$ denote the set of the restriction to $\Omega$ of functions belonging to $\mathcal{C}^{\infty}_{0}(\mathbb{R}^N)$. We have the following theorem relating to~\cite[Theorem~3]{GJP2}.
\begin{theorem}\label{th6.3}
Assume that $\Omega$ satisfies the segment property and  an $N$-function $M$ satisfies \ref{incBM} and \ref{X1}  (resp. \eqref{M>p}, \ref{incBM}, and \ref{X1p}). Then
\begin{itemize}
\item[1)] $\mathcal{C}^{\infty}_{0}(\overline{\Omega})$ is dense in $W^{m}E_M(\Omega)$ with respect to the norm topology in $W^{m}L_M(\Omega)$.
\item[2)] For every $u\in W^{m}L_M(\Omega)$, there exist $\lambda>0$ and a sequence of functions $u_k\in\mathcal{C}^{\infty}_{0}(\overline{\Omega})$ such that  $u_k\mconv u$ in $W^mL_M(\Omega)$.
\end{itemize}
\end{theorem}
In the Orlicz-Sobolev framework, the second result of Theorem \ref{th6.3} was proved by Gossez in \cite{GJP2} assuming that $\Omega$ satisfies additionally the cone property. Such a property in \cite{GJP2} guarantees that any element of $W^{m}L_M(\Omega)$ with compact support in $\overline{\Omega}$ belongs to $W^{m-1}E_M(\Omega)$.
The embedding and approximate results obtained in \cite{DT} allowed  Gossez to prove only the convergence of smooth functions with compact support only for $|\alpha|=m$. Here, we extend the result to the more general setting of Musielak-Orlicz-Sobolev spaces and we enhance it by removing the cone property using \cite[Lemma 4.1]{AY}. Our approach is based on the mean continuity of the translation operator on the set of bounded functions compactly supported in $\Omega$.

\medskip
We give below the extention of~\cite[Theorem 4]{GJP2} by Gossez.
\begin{theorem}\label{densitym0l}
Assume that $\Omega$ satisfies the segment property and  an $N$-function $M$ satisfies \ref{incBM} and \ref{X1}  (resp. \eqref{M>p}, \ref{incBM}, and \ref{X1p}) and let $M^{\ast}$ satisfy \ref{incBM}. Then for every $u\in W^{m}_0L_M(\Omega)$, there exist $\lambda>0$ and a sequence of~functions $u_k\in\mathcal{C}^{\infty}_{0}(\Omega)$ such that $u_k\mconv u$ in $W^mL_M(\Omega)$.
\end{theorem}
The above theorem has the following consequences being an extension to Musielak-Orlicz-Sobolev spaces of~the~results proved by Gossez in \cite[Theorem 1.3]{GJP1} in the case of Orlicz spaces.

\begin{corollary}\label{th5.7}
Assume that $\Omega$ satisfies the segment property and $M\in\Phi$ satisfies \ref{incBM} and \ref{X1}  (resp. \eqref{M>p}, \ref{incBM}, and \ref{X1p}). Then
\begin{itemize}
\item [1)] $\mathcal{C}^{\infty}_{0}(\overline{\Omega})$ is $\sigma(\Pi L_M,\Pi L_{M^{\ast}})$-dense in $ W^{m}L_M(\Omega)$.
\item [2)] $\mathcal{C}^{\infty}_{0}(\Omega)$ is $\sigma(\Pi L_M,\Pi L_{M^{\ast}})$-dense in $W^{m}_{0}L_M(\Omega)$ 
provided that $M^{\ast}$ satisfies the hypothesis \ref{incBM}.
\end{itemize}
\end{corollary}

\subsubsection*{The Musielak-Orlicz spaces and PDEs}

It can be useful in analysis of PDEs (see e.g.~\cite{pgisazg1,gwiazda-ren-ell}) to provide a modular density result not for $W_0^mL_M(\Omega)$, but for \[V_0^mL_M(\Omega)=\{u\in W^{1,1}_0(\Omega)\cap  W^{m,1}(\Omega):\ D^m u\in L_M(\Omega)\},\]
where $W^{1,1}_0(\Omega)$ is defined as a closure of $C_0^\infty(\Omega)$ in $W^{1,1}(\Omega)$--norm.  
 In general,  $W_0^mL_M(\Omega)\neq V_0^mL_M(\Omega).$ Thus, we state below the conjecture. We solve the problem in Appendix for~$m=1$ only.
\begin{conj}\label{conj} 
Assume that $\Omega\subset\rn$ satisfies the segment property and   an $N$-function $M$ satisfies \ref{incBM} and \ref{X1}  (resp. \eqref{M>p}, \ref{incBM}, and \ref{X1p}). Then for every $u\in V^{m}_0L_M(\Omega)$ with $\supp\,u\subset\subset\Omega$, there exist $\lambda>0$ and a~sequence of functions $u_k\in\mathcal{C}^{\infty}_{0}(\Omega)$ such that $D^\alpha u_k\xrightarrow[k\to \infty]{} D^\alpha u$ in~$L^{1}(\Omega
)$ for any $|\alpha|\le m$ and $D^mu_k\xrightarrow[k\to \infty]{M} D^m u$ modularly in $ L_M(\Omega)$.
\end{conj}
Furthermore, in some applications if a construction of a function to approximate is known,  another approach shall be more appropriate. In any reflexive space, e.g. whenever both $M,M^*\in\Delta_2$   and  satisfy \ref{incBM}, Mazur's Lemma ensures the existence of a strongly converging finite affine combination of the weakly converging sequence. Then, if the function is defined as a limit of regular ones, instead of the results of this paper we shall rather approximate it according to Mazur's Lemma as e.g. in~\cite{pgisazg2,pgisazg1}. Let us notice further that, as mentioned in the introductions of ~\cite{pgisazg2,pgisazg1}, the regularity condition (M) is necessary  therein only in the approximation theorems. Whole the proof of existence and uniqueness therein work only under the restriction that $M$ is an $N$-function. Thus, we can replace \cite[(M)]{pgisazg1} with (\ref{incBM} and \ref{X1}) or (\ref{incBM}, \eqref{M>p} and \ref{X1p}).
 
\bigskip 
 
We give below the observation for the spaces equipped with the modular function with the growth at least of~a~power-type. Let us stress that it indicates how sharp the method is, see application in Remark~\ref{rem:double-phase}.
\begin{corollary}\label{coro:d-p}
Assume that $\Omega\subset\subset\rn$ satisfies the segment property and an $N$-function $M$ satisfies~\ref{incBM}, \eqref{M>p} and \ref{X1p}.  Then for every $u\in W^{1,p}_0(\Omega)\cap V^{1}_0L_M(\Omega)$ with $\supp\,u\subset\subset\Omega$, there exists a sequence of functions $u_k\in\mathcal{C}^{\infty}_{0}(\Omega)$ such that $ u_k\xrightarrow[k\to \infty]{} u$ strongly in~$W^{1,p}(\Omega)$ and $D u_k\xrightarrow[k\to \infty]{M} D u$ modularly in $L_M(\Omega)$.
\end{corollary}

\subsection*{Organization of the paper}

 The paper is organized as follows. Section~\ref{section2} supplies preliminaries and necessary  properties of the
Musielak-Orlicz-Sobolev spaces. In Section~\ref{section3} we give several auxiliary lemmata.  Section~\ref{section4} is devoted to the proof of the~main results. In the end we attach Appendix with proofs of examples, lemmata and corollaries.

\section{Background}\label{section2}
In this section  we summarize notation, definitions and properties of the Musielak-Orlicz spaces. For more details we refer to the classical monograph by Musielak \cite{MJ}.

Throughout this paper, we denote by $\Omega$ an open subset of $\mathbb{R}^N$, $N\geq 1$.  We denote by $c$ a various positive constants independent of the variables. Denote $\barint_D f(x)\,dx=\frac{1}{|D|}\int_D f(x)\,dx$, where $|D|$ stands for the Lebesgue measure of the subset $D$ of $\Omega$. 

Define $M^{\ast}: \Omega\times\mathbb{R}^{+}\to\mathbb{R}^{+}$ by
\begin{equation}\label{complementaryfunction}
M^{\ast}(x,s)=\sup_{t\geq 0}\{st-M(x,t)\} \mbox{ for all } s\geq 0\mbox{ and all } x\in \Omega.
\end{equation}The $\Phi$-function $M^{\ast}$ is called the complementary function to $M$ in the sense of Young, the conjugate function, or the Legendre transform. 
It can be checked that if $M$ is an $N$-function, then $M^{\ast}$  is an $N$-function as well. Moreover, we have the Fenchel-Young inequality
\begin{equation}\label{younginequality}
uv\leq M(x,u)+ M^{\ast}(x,v),\quad \forall u,v\geq 0, \forall x\in\Omega.
\end{equation}  

\par For any function $f : \mathbb{R}\to \mathbb{R}$ the second conjugate function $f^{\ast\ast}$ (cf. \eqref{complementaryfunction}), is convex and $f^{\ast\ast}(x) \leq f(x)$. In fact, $f^{\ast\ast}$ is a convex envelope of $f$, namely it is the biggest convex function smaller or equal to $f$.

\begin{definition}[$\Delta_2$-condition]\label{delta2def} 
We say that $M$ satisfies the $\Delta_2$-condition, written $M\in\Delta_2$, if there is a constant $k>0$ such that
\begin{equation}\label{delta2}
M(x,2s)\leq kM(x,s)+h(x)
\end{equation}
for all $s\geq 0$  and almost every  $x\in\Omega$, where $h$ is a nonnegative, integrable function in $\Omega$.
\end{definition}

\begin{lemma}\label{modw}
Let $M$ be an $N$-function and $u_n,u\in L_{M}(\Omega)$. If $u_n\xrightarrow[n\to\infty]{M} u$ modularly, then $u_n\to u$ in $\sigma(L_M,L_{M^{\ast}})$.
\end{lemma}
Nonetheless, for $M\in\Delta_2$, the weak and modular closures are equal. 

\section{Auxiliary results}\label{section3}
We introduce approximate sequences $u_\varepsilon$ and $u_R$.
Let $J$ stands for the Friedrichs mollifier kernel
defined on $\mathbb{R}^{N}$ by
\begin{equation*}\label{eq9.1}
J(x)= ke^{-\frac{1}{1-\|x\|^2}}  \hbox{ if } \|x\|<1 \mbox{ and } 0  \hbox{ if } \|x\|\geq 1,
\end{equation*}
where $k>0$ is such that $\int_{\mathbb{R}^{N}}J(x)dx=1$. For $\varepsilon>0,$ we  define $J_\varepsilon(x)=\varepsilon^{-N}J(x/\varepsilon)$ and  
\begin{equation}\label{eq11}
u_\varepsilon(x)=J_\varepsilon\ast u(x)=\int_{\mathbb{R}^N}J_\varepsilon(x-y)u(y)dy
=\int_{B(0,1)}u(x-\varepsilon y)J(y)dy.
\end{equation}

Define $\chi\in\mathcal{C}^\infty_0(\mathbb{R}^N)$ by $\chi(x)=1$  if  $\|x\|\leq 1$, $\chi(x)=0$ if $\|x\|\geq 2$, $0\leq\chi\leq 1$
and $|D^\alpha\chi(x)|\leq c$  for all $x\in\mathbb{R}^{N}$ and $|\alpha|\leq m$. Given a function $u$, we denote by $u_R$ the function  \begin{equation}
\label{ur} u_R=\chi_R u,\quad \text{ where }\quad\chi_R(x)=\chi(x/R),\  R>0.
\end{equation}

\begin{lemma}\label{densitybmr}
Let $M\in\Phi$. If $u\in W^mL_{M}(\mathbb{R}^N)$ (resp. $u\in W^mE_{M}(\mathbb{R}^N)$) and $u_R$ is given by~\eqref{ur}, then $u_R\mconvr u$ in $W^mL_{M}(\mathbb{R}^N)$   (resp. $u_R\rightarrow u$  in norm in $W^mE_{M}(\mathbb{R}^N)$) as $R\rightarrow \infty$.
\end{lemma}

\begin{lemma}[Lemma 4.1 \cite{AY}]\label{lem4.1}
Let $\Omega$ be an open subset of $\mathbb{R}^N$   and let $M\in\phi$ satisfy~\ref{incBM}. Then the set of bounded compactly supported functions in $\Omega$ is dense in
\begin{enumerate}
\item $E_M (\Omega)$ with respect to the strong topology in $L_M (\Omega)$;
\item $L_M (\Omega)$ with respect to the modular topology in $L_M (\Omega)$.
\end{enumerate}  
\end{lemma}
\begin{lemma}[Lemma 3.1 \cite{AY}]\label{lem3.1} 
Let $\Omega$ be an open subset of $\mathbb{R}^N$   and let $M\in\Phi$ satisfy \ref{incBM}. Then, for every bounded function $u$ with compact support in $\Omega$  and every $\eta$ there exists $h_\eta>0$ such that for all $h$ with $|h|\leq h_\eta$ we have
\[
||\tau_h u -u||_{L_M(\Omega)}\leq \eta.
\]
\end{lemma}


{ 
\begin{lemma}\label{lem1}
Let $\Omega$ be an open subset of $\mathbb{R}^N$ and $M\in\Phi$. Then for every $u\in L_M(\Omega)$ (resp. $u\in L_M(\Omega)\cap L^p(\Omega)$),  
there exists a constant $c>0$ depending on $\|u\|_{L^1(\Omega)}$ (resp. $\|u\|_{L^p(\Omega)}$), but not depending on $\varepsilon$ such that
\[
|u_\varepsilon(x)|\leq c \varepsilon^{-N}\qquad\qquad(\text{resp. }|u_\varepsilon(x)|\leq c\varepsilon^{-\frac{N}{p}}).
\]
\end{lemma}

\begin{lemma}\label{w_epsilon<c/}
Let $\Omega$ be an open subset of $\mathbb{R}^N$ with the segment property and let $M\in\Phi$. Consider $r_i$ and $\ve_i$ related to $\theta_i'$, satisfying~\eqref{eq6.23} and~\eqref{epsi}, respectively. Assume further that $w\in L_M(\Omega)$ (resp. $w\in L_M(\Omega)\cap L^p(\Omega)$) and $J_{\varepsilon_i}\ast(w)_{r_i}(x)$ is given by~\eqref{eq11} and~\eqref{e1}. Then there exists a constant $c>0$ depending on $\|w\|_{L^1(\Omega)}$ (resp. $\|w\|_{L^p(\Omega)}$) such that
\begin{equation}\label{eq12'}
|J_{\varepsilon_i}\ast(w)_{r_i}(x)|\leq c \varepsilon_{i}^{-N}\qquad\qquad (\text{resp. }
|J_{\varepsilon_i}\ast(w)_{r_i}(x)|\leq c \varepsilon_{i}^{-\frac{N}{p}})
\end{equation}
\end{lemma}}

We give below an observation on the regularity of $M$, when we define  \begin{equation}
\label{tildeM}
\widetilde{M}(y,s):=\lim_{\delta\to 0^+} \barint_{B(y,\delta)} M(z,s)\, dz\qquad\text{ and for }\varepsilon>0 \quad \widetilde{M}_{x,\varepsilon}(s):=\inf_{y\in B(x,\varepsilon)} \widetilde{M}(y,s)\end{equation}
 and recall that $(\widetilde{M}_{x,\varepsilon})^{**}$ stands for the second conjugate. 

\begin{lemma}\label{inequalitym**}
Let $\Omega$ be an open subset of $\mathbb{R}^N$ and an $N$-function $M$ satisfy \ref{X1}  (resp. \eqref{M>p} and \ref{X1p}). Let $\varepsilon\in\big(0,1/2\big]$ be arbitrary.
Then, for all $x$, $y\in\Omega$ such that $y\in B(x,\varepsilon/2)$ we have
\begin{equation}\label{m<phim**}
\frac{M(y,s)}{(\widetilde{M}_{x,\varepsilon})^{**}(s)}\leq 4(\varphi(\varepsilon,s))^2.
\end{equation}
\end{lemma}

\begin{lemma}\label{lem3.4}
Let $\Omega$ be an open subset of $\mathbb{R}^N$ 
and an $N$-function $M$ satisfy \ref{X1} (resp. \eqref{M>p} and \ref{X1p}). Let $\varepsilon\in\big(0,1/2\big]$ be arbitrary. Then, for any function $u\in L_M(\Omega)$ (resp. $u\in L_M(\Omega)\cap L^p(\Omega)$), the function $u_\varepsilon\in L_M(\Omega)$. Moreover, the following inequality holds true
\begin{equation}\label{eq1.8.9}
\int_\Omega M(x,u_\varepsilon(x)/\lambda)dx\leq 4(\varphi(\varepsilon,\varepsilon^{-N}c/\lambda))^3\int_\Omega M(x,u(x)/\lambda)dx,
\end{equation}
\begin{equation}\label{eq1.8.9p}
\left(\text{resp.}\quad \int_\Omega M(x,u_\varepsilon(x)/\lambda)dx\leq 4(\varphi(\varepsilon,\varepsilon^{-\frac{N}{p}}c/\lambda))^3\int_\Omega M(x,u(x)/\lambda)dx\right),
\end{equation}
for some $\lambda>0$. The constant $c$ is the one that appears in Lemma \ref{lem1}.
\end{lemma}

{ 
\begin{lemma}\label{M(w_epsilon)<M(w)}
Let $\Omega$ be an open subset of $\mathbb{R}^N$ with the segment property and an $N$-function $M$  satisfies \ref{X1} (resp. \eqref{M>p} and \ref{X1p}).  Consider $r_i$ and $\ve_i$ related to $\theta_i'$, satisfying~\eqref{eq6.23} and~\eqref{epsi}, respectively. Assume further that $w\in L_M(\Omega)$ (resp. $w\in L_M(\Omega)\cap L^p(\Omega)$) and $J_{\varepsilon_i}\ast(w)_{r_i}(x)$ is given by~\eqref{eq11} and~\eqref{e1}. Moreover, the following inequality holds true
\begin{equation}\label{eq1.8.9'}
\int_\Omega M(x,J_{\varepsilon_i}\ast(w)_{r_i}(x)/\lambda)dx\leq 4(\varphi(\varepsilon_{i},\varepsilon_{i}^{-N}c/\lambda))^3\int_\Omega M(x,w(x)/\lambda)dx,
\end{equation}
\begin{equation}\label{eq1.8.9'p}
\left(\text{resp.}\quad  \int_\Omega M(x,J_{\varepsilon_i}\ast(w)_{r_i}(x)/\lambda)dx\leq 4(\varphi(\varepsilon_{i},\varepsilon_{i}^{-\frac{N}{p}}c/\lambda ))^3\int_\Omega M(x,w(x)/\lambda)dx\right),
\end{equation}
for some $\lambda>0$, where $c$ is the constant from Lemma~\ref{w_epsilon<c/}. 
\end{lemma}
}
\begin{lemma}\label{lem:tech}
If $f:\R_+\to\R_+$ is an increasing convex function, then $f\left(\sum_{i=1}^{k}\frac{a_i}{2^i}\right)\leq \sum_{i=1}^{k}\frac{2^{k-i}}{2^k-1} f(a_i)$
\end{lemma}
\begin{proof}
\[f\left(\sum_{i=1}^{k}\frac{a_i}{2^i}\right)=f\left(\sum_{i=1}^{k}\frac{2^{k-i}}{2^k-1}\frac{2^k-1}{2^{k-i}}\frac{a_i}{2^i}\right)\leq \sum_{i=1}^{k}\frac{2^{k-i}}{2^k-1}f\left(\frac{2^k-1}{2^{k-i}}\frac{a_i}{2^i}\right)\leq \sum_{i=1}^{k}\frac{2^{k-i}}{2^k-1}f(a_i).\]
\end{proof}
\par In the following lemma we prove the convergence of the mollification. Such result in the case of Orlicz spaces can be found  in~\cite[Lemma~3.16]{AF} .
\begin{lemma}\label{densityWM1}
Assume that an $N$-function $M$ satisfies \ref{incBM} and \ref{X1} (resp.~\eqref{M>p}, \ref{incBM} and \ref{X1p}).
 If $u\in W^mL_M(\Omega)$  has a~compact support in ${\Omega}$ and $u_\varepsilon$ stands for the sequence defined in \eqref{eq11}, then   $u_\varepsilon\xrightarrow[\varepsilon\to 0^+]{mod} u  \mbox{ in } W^mL_M(\Omega).$
 
  If additionally $u\in W^mE_M(\Omega)$, then $
u_\varepsilon \rightarrow u \mbox{ in norm in  } W^mE_M(\Omega).$
\end{lemma}
\section{Proofs of main results}\label{section4}
\subsubsection*{Proof of Theorem \ref{densityrn}}
Let $u\in W^mL_M(\mathbb{R}^N)$. We shall find $\lambda>0$ and $v\in \mathcal{C}^{\infty}_{0}(\mathbb{R}^N)$ such  that for every $\eta\geq0$
\begin{equation*}
\int_{\mathbb{R}^N} M(x,|D^\alpha u(x)-D^\alpha v(x)|/\lambda)dx\leq\eta \qquad \forall |\alpha|\leq m.
\end{equation*}

According to Lemma \ref{densitybmr}, we can assume that $u$ is compactly supported in $\mathbb{R}^N$. So that by Lemma~\ref{densityWM1}, one can approximate $u$ by a function $v$ in $\mathcal{C}^{\infty}_{0}(\mathbb{R}^N)$ and this yields the result. If $u$ belongs to $W^mE_M(\mathbb{R}^N)$, we prove the result following the similar way noting that $\lambda>0$ can be chosen arbitrary.\qed 
\subsubsection*{Proof of Theorem \ref{th6.3}}
Let $u\in W^mL_{M}(\Omega)$ and fix arbitrary $\eta\geq 0$. We shall find  $v\in {\cal C}^{\infty}_{0}(\overline{\Omega})$ and $\lambda>0$ such that
\begin{equation}\label{eq1.2.5}
 \int_\Omega M(x,|
D^\alpha u-D^\alpha v|/\lambda) dx\leq \eta,\qquad \forall |\alpha|\leq m.
\end{equation}
We will construct $v$ using a finite sequence of functions $v_i\in\mathcal{C}^{\infty}_{0}(\overline{\Omega})$, $i=0,1,\cdots,k$, satisfying
\begin{equation*}
\int_\Omega M(x,|
D^\alpha u-D^\alpha v|/\lambda) dx\leq \int_\Omega M\left(x,\sum_{i=0}^{k}|
D^\alpha u_i-D^\alpha v_i|/\lambda_i^\alpha\right) dx \leq \eta,\qquad \forall |\alpha|\leq m
\end{equation*}
with some $\lambda_i^\alpha>0$. The sequence $\{v_i\}_i$ is related to the covering we introduce below.\\
Without loss of generality we can assume that $\supp\, u\subset K$ and $K$ is compact (see Lemma \ref{densitybmr}).  We will distinguish the two cases:
either easy part $K\subset\subset\Omega$ or hard part $K\cap\partial\Omega\neq\emptyset$. Since there exist a finite open covering of $\overline{\Omega}$ given by the segment property, and $K\cap \partial\Omega$ is compact, there exists also a finite collection $\{\widehat{\theta}_i\}_{i=1}^{k}$ covering  $K\cap\partial\Omega$.
Let $E=K\setminus \cup_{i=1}^{k}\widehat{\theta}_i$. Since $E$ is a compact subset  of $\Omega$, there exists an open set $\widehat{\theta}_0$ with a compact closure in $\Omega$ such that $E\subset\overline{\widehat{\theta}_0}\subset\Omega$. Hence $\{\widehat{\theta}_{i}\}^k_{i=0}$ is an open covering of $K$.
Moreover, as in the proof of \cite[Theorem 1.9]{AG}, we can construct another open covering  $\{\theta^{\prime}_{i}\}^{k}_{i=0}$ of $K$
with $\theta^{\prime}_{i}$ has a~compact closure in $\widehat{\theta}_i$ for $i=0,1,\dots,k$.

\par Let $\{\psi_i\}_{i=1}^k$ be a partition of unity associated to $\{\theta_{i}^{\prime}\}_{i=0}^{k}$, with  $\sum_{i=0}^k\psi_i=1$ on
$K$ and let $u_{i}=u \psi_i$. Then $u=\sum_{i=0}^{k}u_{i}$, $\supp\, u_{i}\subset \theta_{i}^{\prime}$
and $u_{i}$ belongs to $W^mL_M(\Omega)$, for $i=0,1,\dots,k$.
\definecolor{xdxdff}{rgb}{0.49019607843137253,0.49019607843137253,1.}
\definecolor{qqqqff}{rgb}{0.,0.,1.}
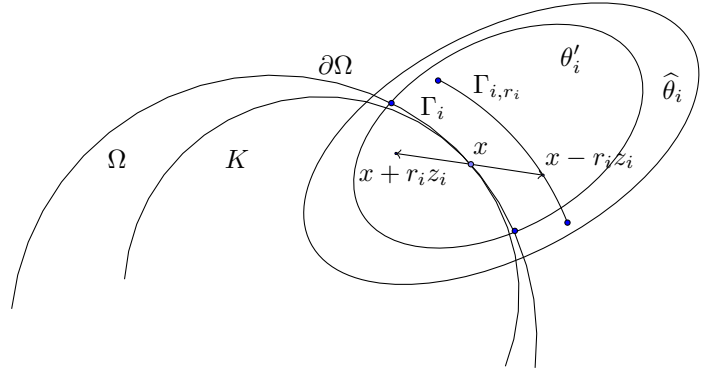
\begin{center}
\begin{tikzpicture}
\draw [shift={(6.27,0.97)}] plot[domain=-0.1122572301875442:3.029335423402249,variable=\t]({1.*3.501728144787942*cos(\t r)+0.*3.501728144787942*sin(\t r)},{0.*3.501728144787942*cos(\t r)+1.*3.501728144787942*sin(\t r)});
\draw [shift={(6.91,1.56)}] plot[domain=-0.3747920417712233:3.0668006118185698,variable=\t]({1.*2.626518608348321*cos(\t r)+0.*2.626518608348321*sin(\t r)},{0.*2.626518608348321*cos(\t r)+1.*2.626518608348321*sin(\t r)});
\draw [rotate around={27.616231926621733:(9.3,3.56)}] (9.3,3.56) ellipse (2.8589835405427113cm and 1.4937492711610365cm);
\draw [rotate around={29.021273224702597:(9.26,3.66)}] (9.26,3.66) ellipse (2.085185072562017cm and 1.2491184038494878cm);
\draw (3.94,3.6) node[anchor=north west] {$\Omega$};
\draw (11.32,4.63) node[anchor=north west] {$\widehat{\theta}_i$};
\draw (9.96,5.) node[anchor=north west] {$\theta_i^\prime$};
\draw (6.74,4.85) node[anchor=north west] {$\partial\Omega$};
\draw (5.5,3.6) node[anchor=north west] {$K$};
\draw (8.1,4.3) node[anchor=north west] {$\Gamma_i$};
\draw [shift={(6.76,1.15)}] plot[domain=0.37848821638109936:1.0850392989415067,variable=\t]({1.*3.6804890979325013*cos(\t r)+0.*3.6804890979325013*sin(\t r)},{0.*3.6804890979325013*cos(\t r)+1.*3.6804890979325013*sin(\t r)});
\draw (8.8,4.6) node[anchor=north west] {$\Gamma_{i,r_i}$};
\draw [->] (8.896019688524623,3.2864888507137353) -- (9.855307477257492,3.141248759254819);
\draw [->] (8.896019688524623,3.2864888507137353) -- (7.9,3.43);
\draw (8.8,3.7) node[anchor=north west] {$x$};
\draw (9.76,3.61) node[anchor=north west] {$x-r_iz_i$};
\draw (7.28,3.39) node[anchor=north west] {$x+r_iz_i$};
\begin{scriptsize}
\draw [fill=qqqqff] (9.48,2.4) circle (1.pt);
\draw [fill=qqqqff] (10.18,2.51) circle (1.pt);
\draw [fill=qqqqff] (8.46,4.4) circle (1.pt);
\draw [fill=qqqqff] (7.84,4.1) circle (1.pt);
\draw [fill=xdxdff] (8.896019688524623,3.2864888507137353) circle (1.pt);
\draw [fill=xdxdff] (9.855307477257492,3.141248759254819) circle (.4pt);
\draw [fill=qqqqff] (7.9,3.43) circle (.4pt);
\end{scriptsize}
\end{tikzpicture}
\end{center}
For a fixed $i$ such that $1\leq i\leq k$, we extend $u_{i}$ to $\mathbb{R}^N$ by zero outside $\theta_i^\prime$.
Let $z_i$ be a nonzero vector associated to $\widehat{\theta}_i$ by the segment property and let $r_i\in(0,1)$ be such that
\begin{equation}\label{eq6.23}
0<r_i<\min\{1/(|z_i|+1),\dist(\theta_{i}^{\prime},\partial\widehat{\theta}_i)|z_i|^{-1}\}.
\end{equation}
Denote $\Gamma_i=\overline{\theta_{i}^{\prime}}\cap \partial \Omega$ and $\Gamma_{i,r_i}=\Gamma_i-r_iz_i$,  then we have $\Gamma_{i,r_i}\subset \widehat{\theta}_i$. Indeed, for $x\in\Gamma_{i,r_i},$ we get by \eqref{eq6.23} $ \dist(x,\theta_{i}^{\prime})\leq \dist(x,\Gamma_{i})\leq |r_iz_i|< \dist(\theta_{i}^{\prime},\partial\widehat{\theta}_{i})$. Furthermore, $\Gamma_{i,r_i}\cap\overline{\Omega}=\emptyset$. Indeed, if $\Gamma_{i,r_i}\cap\overline{\Omega}\neq 0$ then there exists 
$x\in\Gamma_{i,r_i}\cap\overline{\Omega}\subset\widehat{\theta}_{i}\cap\overline{\Omega}$ and $x+r_iz_i\in \Gamma_{i}$
contradicting the segment property.

 If $i=0$, we consider $\supp\, u_{0}\subset\theta^{\prime}_0\subset \Omega$. Choosing $\varepsilon_0>0$ small enough such that $\varepsilon_0<dist(\theta^{\prime}_0,\partial\Omega)$, the regularized function $v_0=J_{\varepsilon_0}*u_{0}$ belongs to $\mathcal{C}^{\infty}_{0}(\Omega)$. We fix $\lambda_0, {\eta_0}>0$, which   due to Lemma \ref{densityWM1} exist such that for all sufficiently small $\varepsilon_{0}>0$ we have
\begin{equation}\label{eqa2.1.7}
\int_\Omega M\Big(x,\frac{|D^\alpha u_{0}(x)-D^\alpha J_{\varepsilon_{0}}*u_{0}(x)|}{\lambda_0}\Big)dx\leq \frac{ {\eta_0}}{2},\qquad \forall |\alpha|\leq m.
\end{equation}

Now, according to the decomposition without loss of generality we can assume that $u$ has its support in a~compact set $K\subset\overline{\Omega}$ with $K\cap \partial\Omega\neq \emptyset$.   By the partition of unity, we arrive to the case $K\subset\widehat{\theta}_i$ for some $i$.  The construction of approximate sequence for a function whose support touches the boundary will be based on the idea of pushing support of $u$ (restricted to a set of a smaller covering) a bit to outside of $\Omega$ to cover its uniform neighbourhood (but still remaining in a set from a bigger covering). Note that for this we exploit the fact that we can consider local systems of coordinates associated with the segment property. We would not be able to do our construction e.g. in the presence of external cusps. Afterwards we mollify and prove convergence.

Fix $1\leq i\leq k$ and define \[(u_{i})_{r_i}(x)=u_{i}(x+r_iz_i).\] Since  $\dist(\Gamma_{i,r_i},\overline{\Omega})>0$, for 
\begin{equation}
\label{epsi}\varepsilon_i<\dist(\Gamma_{i,r_i},\widehat{\theta}_{i}\cap\overline{\Omega})\end{equation}  and extending $u_{i}$ to $\mathbb{R}^N$ to be identically zero outside $\widehat{\theta}_i $, we define the sequence
\begin{equation}\label{e1}
v_i^{\varepsilon_i,r_i}(x)=J_{\varepsilon_i}*(u_{i})_{r_i}(x)=\int_{B(0,1)} J(y)u_i(x+r_iz_i-\varepsilon_i y)dy,
\end{equation}
see \eqref{eq11}.
Our aim is now to estimate \[J^i= \int_\Omega M\left(x,\frac{|D^\alpha u_i(x)-J_{\varepsilon_i}*(D^\alpha u_i)_{r_i}|}{4\lambda_{i,n}^{\alpha }}\right)dx.\] 
Observe that $\supp\, v_i^{\varepsilon_i,r_i}\subset\overline{\Omega}$ and so the function $v_i^{\varepsilon_i,r_i}$ belongs to $\mathcal{C}^{\infty}_{0}(\overline{\Omega})$.
Thus, applying { Lemma~\ref{w_epsilon<c/}} to the function $(u_{i})_{r_i}$, defined on $\overline{\Omega}$ by the segment property, we get $|v_i^{\varepsilon_i,r_i}|\leq\frac{c_i}{\varepsilon_i^N}$ (resp. $|v_i^{\varepsilon_i,r_i}|\leq\frac{c_i}{\varepsilon_i^{N/p}}$)
  where $c_i>0$ depends on the norm $\|u_i\|_{L^1(\supp\; u)}$. 
Note that $D^\alpha u_i\in L_M(\Omega)$ for every $|\alpha|\leq m$ and so there exist $\lambda_i^\alpha$  such that $\int_\Omega M(x,|D^\alpha u_i(x)|/\lambda_i^\alpha)dx<+\infty$. By Lemma~\ref{lem4.1}  there exist a sequence of functions
$\{u_{i,n}^\alpha\}$, where $u_{i,n}^\alpha$ is bounded and has a compact support in $\Omega$, and $\lambda_{i}^{\alpha,\eta_0}>0$
and $n_{i}^{\alpha,\eta_0}>0$, such that for every $n>n_{i}^{\alpha,\eta_0}$ \begin{equation}\label{eqa2.2.2}
I_1=\int_\Omega M\Big(x,\frac{|D^\alpha u_i(x)-u_{i,n}^\alpha(x)|}{\lambda_{i}^\alpha}\Big)dx\leq \eta_0.
\end{equation} 
The Jensen inequality yields 
\begin{equation}
\label{1stJensen}
\begin{split}J^i &\leq\frac{1}{4}\int_\Omega M\left(x,\frac{|D^\alpha u_i(x)-u^\alpha_{i,n}(x)|}{ \lambda_{i}^{\alpha}}\right)dx+\frac{1}{4}\int_\Omega M\left(x,\frac{|u^\alpha_{i,n}(x)- (u^\alpha_{i,n}(x))_{r_i}|}{ \lambda_{i}^{\alpha}}\right)dx+\\
&+\frac{1}{4}\int_\Omega M\left(x,\frac{|(u^\alpha_{i,n}(x))_{r_i}-J_{\varepsilon_i}*(u^\alpha_{i,n}(x))_{r_i}|}{ \lambda_{i}^{\alpha}}\right)dx+\frac{1}{4}\int_\Omega M\left(x,\frac{|J_{\varepsilon_i}*(u^\alpha_{i,n}(x))_{r_i}-J_{\varepsilon_i}*(D^\alpha u_i(x))_{r_i}|}{ \lambda_{i}^{\alpha}}\right)dx\\&=\frac{1}{4}(I_1+I_2+I_3+I_4).\end{split}
\end{equation}
We shall show that $J^i\leq C_i^\alpha\eta_0$ for some constant $C_i^\alpha>0$ whenever $\varepsilon_i$ and $r_i$ are sufficently small. 

The case of $I_1$ is done by~\eqref{eqa2.2.2}. We deal with $I_2$ using Lemma~\ref{lem3.1}, which ensures that there exist $r_{i,n}^{\alpha,\eta_0}>0$, such that for all $r_i<r_{i,n}^{\alpha,\eta_0}$ we can estimate $I_2\leq \eta_0$. In order to find a bound on $I_3$ we notice that the Jensen inequality and the Fubini theorem yield
\begin{equation}\label{eqa2.2.1}
\begin{array}{lll}
I_3&=&\int_{\Omega} M\Big(x,\frac{| (u_{i,n }^\alpha(x))_{r_{i }}- J_{\varepsilon_i}\ast  (u_{i,n}^{\alpha}(x))_{r_{i }}|}{\lambda_{i}^{\alpha}}\Big)dx\\
&=& \int_{\Omega} M\Bigg(x,\Bigg|\int_{B(0,1)}\frac{J (y)}{\lambda_{i}^{\alpha}} \Big(u_{i,n }^\alpha(x+r_{i }z_i)-u_n^\alpha(x+r_{i }z_i-\varepsilon_iy)\Big)dy\Bigg|\Bigg)dx\\
&\leq& \int_{B(0,1)} J(y)\int_{\Omega} M\Big(x, \frac{\Big|u_{i,n }^\alpha(x+r_{i }z_i)-u_n^\alpha(x+r_{i }z_i-\varepsilon_iy)\Big|}{\lambda_{i}^{\alpha}}\Big)dxdy,
\end{array}
\end{equation}
where $r_i<r_{i,n}^{\alpha,\eta_0}$ and its relation with $z_i$ is given by~\eqref{eq6.23}.  Using Lemma~\ref{lem3.1}, for every $|y|< 1$ and $\eta_0>0$ there exists $\varepsilon_{i,n}^{\alpha,\eta_0}$ such that for $\varepsilon_{i}\leq\varepsilon_{i,n}^{\alpha,\eta_0}$ we get
\begin{equation*}
\int_{\Omega} M\Big(x, \frac{u_{i,n }^\alpha(x+r_{i }z_i)-u_{i,n }^\alpha(x+r_{i }z_i-\varepsilon_iy)}{\lambda_{i }^{\alpha}}\Big)dx\leq\eta_0,
\end{equation*}
thus via~\eqref{eqa2.2.1} we can estimate $I_3\leq \eta_0 \int_{B(0,1)} J(y) dy=\eta_0 $ (whenever $r_i<r_{i,n}^{\alpha,\eta_0}$ and $\varepsilon_i<\varepsilon_{i,n}^{\alpha,\eta_0}$). In the case of $I_4$ we apply {  Lemma~\ref{M(w_epsilon)<M(w)}} to the expression $(u_{i,n}^\alpha(x))_{r_{i}}-(D^\alpha u_i(x))_{r_i}$   and \eqref{eqa2.2.2} to claim that there exist $\overline{\varepsilon}_{i,n}^{\alpha,\eta_0}$, such that for all $\varepsilon_i<\overline{\varepsilon}_{i,n}^{\alpha,\eta_0}$ (and $r_i<r_{i,n}^{\alpha,\eta_0}$) we have
\begin{equation}\label{eqa2.2.3}
\begin{array}{lll}
I_4&=&\int_\Omega M\Bigg(x,\frac{\Big|J_{\varepsilon_i}\ast  (u_{i,n }^\alpha(x))_{r_{i }}-J_{\varepsilon_i}\ast(D^{\alpha}u_{i}(x))_{r_{i }}\Big|}{\lambda_{i }^{\alpha}}\Bigg)dx\\
&\leq& 4\Big(\varphi\Big(\varepsilon_i,\frac{c_i}{\lambda_{i}^\alpha\varepsilon_i^{N}}\Big)\Big)^3\int_\Omega M\Big(x,\frac{\big|u_{i,n}^\alpha(x)-D^{\alpha}u_{i}(x)\big|}{\lambda_{i }^{\alpha}}\Big)dx\ \leq\  4\Big(\varphi\Big(\varepsilon_i,\frac{c_i}{\lambda_{i }^{\alpha}\varepsilon_i^{N}}\Big)\Big)^3\eta_0.
\end{array}
\end{equation}
To sum up, we get in \eqref{1stJensen}  for $\varepsilon_{i}\leq\min\{\varepsilon_{i,n}^{\alpha,\eta_0},\overline{\varepsilon}_{i,n}^{\alpha,\eta_0}\}$ and $r_i<r_{i,n}^{\alpha,\eta_0}$
\begin{equation*}\label{} 
\int_\Omega M\Big(x,\frac{|D^{\alpha}u_i(x)-D^{\alpha}v_{i}(x)|}{4\lambda_{i }^{\alpha}}\Big)dx \leq \frac{1}{4}( I_1+I_2+I_3+I_4) \leq  \eta_0\left(\frac{3}{4}+\Big(\varphi\Big(\varepsilon_i,\frac{c_i}{\lambda_{i }^{\alpha}\varepsilon_i^{N}}\Big)\Big)^3 \right), 
\end{equation*}
for every $\alpha$  such that $ |\alpha|\leq m$. Therefore, for an arbitrary $\bar{\eta}>0$ and sufficiently small $\varepsilon_{i}^{\bar{\eta}}$ and $r_i^{\bar{\eta}}$ we have
\begin{equation*}
J=\int_\Omega M\Big(x,\frac{|D^{\alpha}u_i(x)-D^{\alpha}J_{\varepsilon_{i }^{\bar{\eta}}}\ast (u_i)_{r_{i }^{\bar{\eta}}}(x)|}{4\lambda_{i }^{\alpha }}\Big)dx\leq \frac{\bar\eta}{2^i}.
\end{equation*}
Let $v=\sum_{i=1}^{k}J_{\varepsilon_{i}^{\bar{\eta}}}\ast (u_i)_{r_{i}^{\bar{\eta}}}+J_{\varepsilon_{0}}*u_0$. and note that $v$ belongs to $\mathcal{C}^{\infty}_{0}(\overline{\Omega})$. Take $\lambda=\max\{\lambda_0,4\lambda_{i }^{\alpha}\}$. By \eqref{eqa2.1.7}, Lemma~\ref{lem:tech}, and the last inequality  for every $\alpha$ such that $ |\alpha|\leq m$ we obtain 
\begin{equation*}
\begin{array}{lll}
&&\int_\Omega M\Big(x,\frac{|D^\alpha u(x)-D^\alpha v(x)|}{2^{k+1}\lambda}\Big)dx \\ 
&\leq& \frac{1}{2}\int_{\Omega}M\Big(x,\sum_{i=1}^k \frac{|D^\alpha u_{i}(x)-D^\alpha J_{\varepsilon_{i}^{\bar{\eta}}}\ast (u_i)_{r_{i}^{\bar{\eta}}}(x)|}{2^{i}\lambda}\Big)dx
+\frac{1}{2^{k+1}}\int_\Omega M\Big(x,\frac{|D^{\alpha}u_0(x)-J_{\varepsilon_0}*D^{\alpha}u_{0}(x)}{\lambda}\Big)dx\\  
&\leq&\sum_{i=1}^{k}\frac{2^{k-i}}{2^k-1}\int_{\Omega}M\Big(x,\frac{|D^\alpha u_{i}(x)-D^\alpha J_{\varepsilon_{i}^{\bar{\eta}}}\ast (u_i)_{r_{i}^{\bar{\eta}}(x)}|}{ \lambda }\Big)dx
+\frac{1}{2^{k+1}}\int_\Omega M\Big(x,\frac{|D^{\alpha}u_0(x)-J_{\varepsilon_0}*D^{\alpha}u_{0}(x)}{\lambda }\Big)dx\\\\
&\leq& \sum_{i=1}^{k}\frac{2^{k-i}}{2^k-1}\frac{1}{2^i}\bar{ \eta}+\frac{1}{2^{k+2}}\bar{ \eta}.
\end{array}
\end{equation*}
Choosing $\bar{\eta}=\frac{\eta}{2}\left(\sum_{i=1}^{k}\frac{2^{k-i}}{2^k-1}\right)^{-1}$ we get~\eqref{eq1.2.5} and hence the second assertion is proven. Indeed, the method of construction gives us a unique $v$ independent of $\alpha$.

To get the first assertion, it suffices to note that if additionally to the above reasoning $u$ belongs to $W^mE_M(\Omega)$, we obtain~\eqref{eq1.2.5} with arbitrary $\lambda>0$ (cf. Lemma~\ref{lem:modular-norm}).\qed 
\subsubsection*{Proof of Theorem \ref{densitym0l}}
Let $u\in W^m_0L_{M}(\Omega)$. Our goal is to find  $v\in {\cal C}^{\infty}_{0}(\Omega)$ such that for some $\lambda>0$ and all $\eta\geq 0$ we get
\begin{equation*} 
\int_\Omega M\Big(x,\frac{|D^\alpha u(x)-D^\alpha v(x)|}{\lambda}\Big)dx\leq \eta \qquad \forall |\alpha|\leq m.
\end{equation*}
The proof follows exactly the same lines as the one of the second claim of Theorem \ref{th6.3}, except that we change $r_i$ by $-r_i$ in \eqref{e1} so that
$$
v_i=J_{\varepsilon_i}*(u_{i})_{-r_i}=\int_{B(0,1)} J(y)u_i(x-r_iz_i-\varepsilon_i y)dy
$$
and choose
$$
\varepsilon_i<dist\bigg((\theta^\prime_i\cap \overline{\Omega})+r_iz_i, \mathbb{R}^N\setminus\Omega\bigg).
$$
Therefore, $\supp\,v_i\subset\Omega$ and the function $v_i$ belongs to $\mathcal{C}^{\infty}_{0}(\Omega)$.\qed
\section*{Appendix}

\subsubsection*{Proofs of Examples}

\begin{enumerate}
\item This case is a direct consequence of \ref{X1}.

\item $M(x,s)=|s|^{p(x)}$ satisfies the \ref{X1} condition with
\begin{equation*}
\varphi(\tau,s)= s^{\sigma(\tau)} \  \hbox{ if }\ s\geq 1 \quad\mbox{ and }\quad
\varphi(\tau,s)=  s^{-\sigma(\tau)} \  \hbox{ if }\ s< 1,
\end{equation*}
where $\sigma : (0, {1}/{2}]\to\mathbb{R}^+$ with $\limsup_{\varepsilon\to 0}\sigma(\varepsilon)=0$. 
 
\item If $M(x,s)=s^{p}+a(x)s^{q}$ and $a\in C^{0,\alpha}$, we have
\[\begin{split}
\frac{M(x,s)}{M(y,s)}&=\frac{s^{p}+a(x)s^{q}}{s^{p}+a(y)s^{q}}=\frac{s^{p}+a(y)s^{q}+(a(x)-a(y))s^{q}}{s^{p}+a(y)s^{q}}=1+\frac{a(x)-a(y)}{s^{p-q}+a(y)}\leq 1+C_a \frac{|x-y|^\alpha}{s^{p-q}+a(y)}\leq\\
&\leq 1+C_a |x-y|^\alpha s^{q-p}.
\end{split}\]Then \ref{X1} holds with $\varphi(\tau,|s|)= C_a\tau^\alpha|s|^{q-p}+1$ if and only if $\alpha+N(p-q)\leq 0$, that is $q\leq p+\alpha/N$. See Remark~\ref{rem:double-phase} to get the sharp range via  \ref{X1p}.

\item Note first that
$$
\frac{M(x,s)}{M(y,s)}=\frac{p(y)}{p(x)}s^{p(x)-p(y)}\leq \frac{p^+}{p^-}s^{p(x)-p(y)}.
$$
So we have to discuss two cases: when $s\geq 1$ and $s\leq 1$.

\textbf{If $s\geq 1$ }
\begin{enumerate}
\item If $p(x)-p(y)\geq 0$ then $s^{p(x)-p(y)}=s^{|p(x)-p(y)|}\leq s^{\frac{c}{\log\frac{1}{|x-y|}}}$.
\item If $p(x)-p(y)\leq 0$ then $s^{p(x)-p(y)}=s^{-|p(x)-p(y)|}$. Since $ {1}/{s}\leq 1 $\\
$$
\frac{1}{s^{\frac{c}{\log\frac{1}{|x-y|}}}}\leq \frac{1}{s^{|p(x)-p(y)|}}\Rightarrow s^{-|p(x)-p(y)|}\leq s^{\frac{c}{\log\frac{1}{|x-y|}}}
$$
Hence
$$
\frac{M(x,s)}{M(y,s)}\leq \frac{p^+}{p^-}s^{\frac{c}{\log\frac{1}{|x-y|}}}, \qquad \forall s\geq1.
$$
\end{enumerate}
\textbf{If $s\leq 1$ }
\begin{enumerate}
\item If $p(x)-p(y)\geq 0$ then $s^{p(x)-p(y)}\leq s^{\frac{-c}{\log\frac{1}{|x-y|}}}$. Indeed,
$$
|p(x)-p(y)|\leq \frac{c}{\log\frac{1}{|x-y|}} \Leftrightarrow \frac{-c}{\log\frac{1}{|x-y|}}\leq p(x)-p(y)\leq \frac{c}{\log\frac{1}{|x-y|}}
$$
then since ${1}/{s}\geq 1$ we can write
\begin{equation*}
\begin{array}{lllll}
 \Big(\frac{1}{s}\Big)^{p(x)-p(y)+\frac{c}{\log\frac{1}{|x-y|}}}\geq1 
&\Rightarrow& \Big(\frac{1}{s}\Big)^{p(y)-p(x)}\leq\Big(\frac{1}{s}\Big)^\frac{c}{\log\frac{1}{|x-y|}}
&\Rightarrow& s^{p(x)-p(y)}\leq s^\frac{-c}{\log\frac{1}{|x-y|}}.
\end{array}
\end{equation*}
\item If $p(x)-p(y)\leq 0$ then as $ {1}/{s}\geq1$, one has
$\frac{1}{s^{|p(x)-p(y)|}}\leq \frac{1}{s^{\frac{c}{\log\frac{1}{|x-y|}}}}$,
which yields $s^{p(x)-p(y)}\leq s^{\frac{-c}{\log\frac{1}{|x-y|}}}.$
Thus
$$
\frac{M(x,s)}{M(y,s)}\leq \frac{p^+}{p^-}s^{\frac{-c}{\log\frac{1}{|x-y|}}}, \qquad \forall s\leq1.
$$
\end{enumerate}

\item 
We compute
	$$
\begin{array}{lll}
\frac{M(x,s)}{M(y,s)}&=&\frac{\sum_{i=1}^k k_i(x)M_i(s)+M_0(x,s)}{\sum_{j=0}^k k_j(y)M_j(s)+M_0(y,s)}
\leq\frac{\sum_{i=1}^k  \vp_i(|x-y|)k_i(y)M_i(s)}{\sum_{j=0}^k k_j(y)M_j(s)}+\frac{M_0(x,s)}{M_0(y,s)}\\\\
&\leq&\sum_{i=1}^k \vp_i(|x-y|)\frac{k_i(y)M_i(s)}{\sum_{j=1}^k k_j(y)M_j(s)}+\vp_0(|x-y|,|s|)\\\\
&\leq&\sum_{j=1}^k \vp_j(|x-y|) \frac{\sum_{i=1}^k  k_i(y)M_i(s)}{\sum_{j=1}^k k_j(y)M_j(s)}+\vp_0(|x-y|,|s|)\\\\
&=&\sum_{j=1}^k \vp_j(|x-y|)+\vp_0(|x-y|,|s|) =\varphi(|x-y|, |s| ).
\end{array}
$$

\end{enumerate}

\subsubsection*{Proof of Remark \ref{rem:double-phase}}
  If $M(x,s)=s^{p}+a(x)s^{q}$ and $1<p< q<\infty$, then
  \[M(x, 2s )=|2s|^{p}+a(x)|2s|^{q}\leq 2^q(|s|^{p}+a(x)|s|^{q})=2^q M(x,s),\]
so $M\in\Delta_2$. Since $M^*$ has the same type, we infer that $M^*\in\Delta_2$.

Note, that the form of $\varphi$ is computed in the proof of point 3. of Examples above. Then \ref{X1p} holds with $\varphi(\tau,|s|)= C_a\tau^\alpha|s|^{q-p}+1$ if and only if $\alpha+N(p-q)/p\leq 0$, that is $q/p\leq 1+\alpha/N $.  \qed

\subsubsection*{Proof of Corollary \ref{th5.7}}
It is a direct consequence of Theorem \ref{th6.3}, Theorem \ref{densitym0l}, and Lemma \ref{modw}.\qed

 
\subsubsection*{Proof of Lemma \ref{modw}}
There exists $\lambda>0$ such that $\int_\Omega M\Big(x,\frac{|u_n(x)- u(x)|}{\lambda}\Big)dx\to 0$ as $n\to\infty$. Thus, $M\Big(x,\frac{|u_n- u|}{\lambda}\Big)$ tends to $ 0$ strongly in $L^1(\Omega)$ and so for a subsequence, still indexed by $n$, $u_n\to u$ a.e. in $\Omega$. For an arbitrary $v\in L_{M^{\ast}}(\Omega)$, there exists $\lambda_v>0$ such that $M^{\ast}\Big(x,\frac{|v|}{\lambda_v}\Big)\in L^1(\Omega)$. Young's inequality allows us to write 
$$
\frac{1}{\lambda\lambda_v}|(u_n-u)v|\leq M\Big(x,\frac{|u_n(x)- u(x)|}{\lambda}\Big)+M^{\ast}\Big(x,\frac{|v(x)|}{\lambda_v}\Big).
$$
Thus, applying Vitali's theorem we obtain $\int_\Omega(u_n-u)vdx\to 0$.\qed
\subsubsection*{Proof of Lemma~\ref{densitybmr} }
For $u\in W^mL_{M}(\mathbb{R}^N)$, there exists $\lambda>0$ such that $\int_{\mathbb{R}^N}M(x,|D^\alpha u|/\lambda)dx<\infty,\; \forall |\alpha|\leq m.$ Observe first that $u_R\in W^mL_{M}(\mathbb{R}^N)$.
On one hand, when $R\to\infty$ one has
$$
M(x,|D^\alpha u-\chi_R(x)D^\alpha u|/2\lambda)\rightarrow 0\quad \mbox{  a.e. in } \mathbb{R}^N,
$$
meanwhile on the other hand
$$
M(x,|D^\alpha u-\chi_R(x)D^\alpha u|/2\lambda) \leq M(x,|D^\alpha u|/\lambda)\in L^1(\mathbb{R}^N).
$$
So that by the Lebesgue Dominated Convergence Theorem, we obtain
\begin{equation}\label{eq2.1.4}
\lim_{r\rightarrow\infty}\int_{\mathbb{R}^N}M(x,|D^\alpha u-\chi_R(x)D^\alpha u|/2\lambda)= 0.
\end{equation}
Now, for $t=\sum_{\beta\neq0, \beta\leq \alpha}\big(^\alpha_\beta\big)$ and by using the Leibnitz formula for $|\alpha|\leq m$ we get
\begin{equation*}
\begin{array}{lll}
I&=&\int_{\mathbb{R}^N} M(x,|D^\alpha u-D^\alpha u_{R}|/4c\lambda t)dx\\
&\leq& \frac{1}{2ct}\int_{\mathbb{R}^N} M\Big(x,|D^\alpha u-\chi_R(x)D^\alpha u|/2\lambda\Big)dx
+\frac{1}{2}\int_{\mathbb{R}^N} M\Big(x,\frac{1}{2ct\lambda }\sum_{\beta\neq0, \beta\leq \alpha}\big(^\alpha_\beta\big)\frac{1}{R^{|\beta|}}\left|(D^\beta\chi)(x/R) D^{\alpha-\beta} u\right|\Big)dx\\
&\leq& \frac{1}{2ct}\int_{\mathbb{R}^N} M\Big(x,\left|D^\alpha u-\chi_R(x)D^\alpha u\right|/2\lambda\Big)dx
+\frac{1}{2t}\sum_{\beta\neq0, \beta\leq \alpha}\big(^\alpha_\beta\big)\int_{\mathbb{R}^N} M\Big(x,\frac{1}{2c\lambda R^{|\beta|}}\left|(D^\beta\chi)(x/R) D^{\alpha-\beta} u\right|\Big)dx.
\end{array}
\end{equation*}
By using \eqref{eq2.1.4} the first term in  the right hand side of the last inequality tends to zero as $R\rightarrow \infty$, while for the second term in the right-hand side  we have
$$
M\Big(x,\frac{1}{2c\lambda R^{|\beta|}}\left|(D^\beta\chi)(x/R) D^{\alpha-\beta} u\right|\Big)\rightarrow 0\quad \text{ as }\quad R\rightarrow\infty\quad \mbox{  a.e. in } \mathbb{R}^N
$$
and for $R>1$
$$
\int_{\mathbb{R}^N} M\Big(x,\frac{1}{2c\lambda R^{|\beta|}}\left|(D^\beta\chi)(x/R) D^{\alpha-\beta} u\right|\Big)dx\leq
\int_{\mathbb{R}^N} M\Big(x,\frac{|D^{\alpha-\beta} u|}{\lambda}\Big)dx<\infty.
$$
Hence by using the Lebesgue Dominated Convergence Theorem, we obtain
$$
\int_{\mathbb{R}^N} M\Big(x,\frac{1}{2c\lambda R^{|\beta|}}\left|(D^\beta\chi)(x/R) D^{\alpha-\beta} u\right|\Big)dx \rightarrow 0 \quad \text{ as }\quad R\rightarrow\infty .
$$
This yields the result for  $u\in W^mL_M(\mathbb{R}^N)$. Analogically we get the norm convergence in $W^mE_{M}(\mathbb{R}^N)$  (cf.~Lemma~\ref{lem:modular-norm}).

\subsubsection*{Proof of Lemma~\ref{lem1}}
We can write
\begin{equation*}
\begin{array}{lll}
|u_\varepsilon(x)|&=&\big|\int_{B(0,1)}u(x-\varepsilon y)J(y)dy\big|
\leq\max_{|y|\leq 1}|J(y)|\int_{B(0,1)} |u(x-\varepsilon y)|dy\\
&\leq& \varepsilon^{-N}\max_{|y|\leq 1}|J(y)|\int_{\Omega\cap B(0,\varepsilon)} |u(x-y)|dy\leq \varepsilon^{-N}\max_{|y|\leq 1}|J(y)|\int_{\Omega\cap B(x,\varepsilon)}  |u(y)|dy\leq\frac{c}{\varepsilon^{N}}.
\end{array}
\end{equation*}
where $c=\max_{|y|\leq 1}|J(y)|\|u\|_{L^1(\Omega)}$.

As for the assertion for $u\in L^p(\Omega)$ we have
\begin{equation*}
\begin{array}{lll}
|u_\varepsilon(x)|&=&\big|\int_{B(0,1)}u(x-\varepsilon y)J(y)dy\big|
\leq \left(\int_{B(0,1)} |u(x-\varepsilon y)|^p dy\right)^\frac{1}{p}\left(\int_{B(0,1)} |J(y)|^{p'} dy\right)^\frac{1}{p'}\leq \\
&\leq& c   \left(\varepsilon^{-N}\int_{ \Omega\cap B(0,\varepsilon)} |u(x-\varepsilon y)|^p dy\right)^\frac{1}{p}\leq c \varepsilon^{-N/p}  \left(\int_{\Omega\cap B(x,\varepsilon)} |u(x-y)|^pdy\right)^\frac{1}{p}\leq c \varepsilon^{-N/p}.
\end{array}
\end{equation*}
where $c=\left(\int_{B(0,1)} |J(y)|^{p'} dy\right)^\frac{1}{p'}\|u\|_{L^p(\Omega)}$.\qed

\subsubsection*{Proof of Lemma~\ref{w_epsilon<c/}}
In the spirit of the proof of Lemma~\ref{lem1} we notice that\begin{equation*}
\begin{array}{lll}
|J_{\varepsilon_i}\ast(w)_{r_i}(x)|&=&\big|\int_{B(0,1)}w(x+r_{i}z_i-\varepsilon_iy)J(y)dy\big|\\
&\leq&\max_{|y|\leq 1}|J(y)|\int_{B(0,1)} |w(x+r_{i}z_i-\varepsilon_iy)|dy\\
&\leq& \varepsilon_{i}^{-N}\max_{|y|\leq 1}|J(y)|\int_{\Omega} \chi_{B(0,\varepsilon_i)}|w(x+r_{i}z_i-y)|dy\\
&\leq& \varepsilon_{i}^{-N}\max_{|y|\leq 1}|J(y)|\int_{\Omega} \chi_{B(x+r_{i}z_i,\varepsilon_i)}
|w(y)|dy\\
&\leq&{c}{\varepsilon_{i}^{-N}}.
\end{array}
\end{equation*}
where $c=\max_{|y|\leq 1}|J(y)|\|w\|_{L^1(\Omega)}$. We note here that for $r_i$ and $\varepsilon_i$ small enough, the segment property ensures that the ball  $B(x+r_{i}z_i,\varepsilon_i)$ is contained in $\Omega$. Applying the H\"older inequality as in the proof of Lemma~\ref{lem1} we get the second assertion.
\qed

\subsubsection*{Proof of Lemma~\ref{inequalitym**}}
Note that for a.e. $y\in \Omega$ and $s\in  \R_+\cup\{0\}$, we have
\begin{equation}\label{ae=}
\wbM(y,s)= {M}(y,s).
\end{equation}
Then from \ref{X1}, for all $x,y\in\Omega$ such that $|x-y|\leq \frac{1}{2}$ one has
\begin{equation}\label{log-hold-Mbar}
\wbM(x,s)\leq \vp(|x-y|,s)\wbM(y,s)
\end{equation}
Moreover, $\wbM$ is locally Lipschitz with respect to $s$. Indeed, for $s\in (0,R)$, $R>0$, we can write
 \[
 \sup_{y\in B(x,\frac{\ve}{2}),\,s<R}\left|\frac{\partial M}{\partial s}(y,s)\right| \leq \sup_{y\in B(x,\frac{\ve}{2})}\frac{M(y,R)-M(y,0)}{R-0}\leq \sup_{y\in B(x,\frac{\ve}{2})}\frac{M(y,R) }{R}.
\]
By virtue of \ref{X1}, we have 
$$
\sup_{y\in B(x, {\ve}/{2})}M(y,R)\leq \varphi(\ve,R)M(x,R).
$$
So that we obtain
\[
\sup_{y\in B(x, {\ve}/{2}),\,s<R}\left|\frac{\partial M}{\partial s}(y,s)\right|
\leq  \frac{\varphi(\frac{1}{2},R)M(x,R)}{R}.
\]
Thus, $\wbM$ and $\Mb$ are continuous in $s$. Fix arbitrary $y\in\overline{B(x, {\ve}/{2})}$. Then we estimate
\begin{equation}\label{A:def:est}
A:=\frac{M(y,s)}{(\Mb)^{**}(s)}\leq 4(\vp(\ve,s))^2.
\end{equation}
Let us start with writing
\[A=\frac{M(y,s)}{ \Mb (s)}\cdot \frac{\Mb (s)}{(\Mb)^{**}(s)}=A_1\cdot A_2 \]
and noting that for any fixed $s\neq 0$, we consider $\{y_s^n\}_n\subset
{\overline{B\Big(x, {\ve}/{2}\Big)}}$, such that for every $n>n(s)$ we have
\[\Mb(s)\geq \wbM(y_s^n,s)-\frac{1}{n}.\]
If necessary taking bigger $n$, we can further estimate
\begin{equation}\label{Mxest}
\Mb(s)\geq \frac{1}{2}\wbM(y_s^n,s).
\end{equation}
Therefore, for a.e. $y\in\overline{B\Big(x,{\ve}/{2}\Big)}$ we have
\begin{equation}
\label{A1}
A_1=\frac{M(y,s)}{ \Mb (s)} =\frac{\wbM(y,s)}{ \Mb (s)}\leq 2\frac{\wbM(y,s)}{ \wbM(y_s^n,s) }\leq 2 \vp(|y-y_s^n|,s)\leq 2\vp(\ve,s),
\end{equation}
due to \eqref{ae=}, \eqref{log-hold-Mbar}, \eqref{Mxest} and monotonicity of $\vp$. As for $A_2$, let us remark that if $\Mb$ is convex in $s$, then $\Mb=(\Mb)^{**}$ and $A_2=1$. Otherwise there exist $s_1<s_2$, such that for every $s\in(s_1,s_2)$ we have $\Mb(s)>(\Mb)^{**}(s)$ and $\Mb(s_i)=(\Mb)^{**}(s_i),$ $i=1,2$. Then for every $t\in[0,1]$ we have
\[
(\Mb)^{**}(ts_1+(1-t)s_2)=t\Mb(s_1)+(1-t)\Mb(s_2).\]
 Let us consider $\{y^n_{s_1}\}_n$, $\{y^n_{s_2}\}_n$ defined similarly to $\{y^n_{s}\}_n$ and estimate
 \[(\Mb)^{**}(ts_1+(1-t)s_2)\geq t\wbM(y_{s_1}^n,s_1)+(1-t)\wbM(y_{s_2}^n,s_2)-\frac{1}{n}.\]
 We can assume without loss of generality that
\[
\wbM(y_{s_1}^n,s_1)< \wbM(y_{s_2}^n,s_1)
\]
because otherwise we arrive at $\wbM\leq (\wbM)^{**}$ that is $A_2=1$.
Hence,
\[
\begin{split}
A_2&=\frac{\Mb (ts_1+(1-t)s_2)}{(\Mb)^{**}(ts_1+(1-t)s_2)}\leq\frac{\wbM (y_{s_2}^n,ts_1+(1-t)s_2)}{t\wbM(y_{s_1}^n,s_1)+(1-t)\Mb(y_{s_2}^n,s_2)-\frac{1}{n}}\leq\\
&\leq \frac{t\wbM (y_{s_2}^n, s_1 )+(1-t)\wbM (y_{s_2}^n, s_2)}{t\wbM(y_{s_1}^n,s_1)+(1-t)\wbM(y_{s_2}^n,s_2)-\frac{1}{n}}=h(t).
\end{split}
\]
Since for $t\in (0,1)$ we notice that
\begin{equation*}
h'(t)=\frac{(\wbM(y_{s_2}^n,s_1)-\wbM(y_{s_1}^n, s_1))\wbM(y_{s_2}^n,s_2)}{(t(\wbM(y_{s_1}^n,s_1)-\wbM(y_{s_2}^n,s_2))+\wbM(y_{s_2}^n,s_2))^2}+\frac{(\wbM(y_{s_2}^n,s_2)-\wbM(y_{s_2}^n,s_1)) }{n(t(\wbM(y_{s_1}^n,s_1)-\wbM(y_{s_2}^n,s_2))+\wbM(y_{s_2}^n,s\xi_2))^2}>0,
\end{equation*}
the maximum of $h$ is attained at $t=1$, which implies
\[A_2\leq \frac{ \wbM (y_{s_2}^n, s_1 ) }{ \wbM(y_{s_1}^n,s_1)-\frac{1}{n}}.
\]
We can restrict ourselves to $n$ sufficiently big to have
\begin{equation}\label{A2}A_2\leq 2\frac{ \wbM (y_{s_2}^n, s_1 ) }{ \wbM(y_{s_1}^n,s_1) }\leq 2\vp (|y_{s_2}^n-y_{s_1}^n|,s_1)\leq 2\vp (\ve,s_1)\leq 2\vp (\ve,s).
\end{equation}
Note that we applied here \eqref{log-hold-Mbar}. Combining~\eqref{A1} with~\eqref{A2} gives~\eqref{A:def:est}.\qed
\subsubsection*{Proof of Lemma~\ref{lem3.4}}
For $u\in L_M(\Omega)$ there exists $\lambda>0$ such that $\int_\Omega M(y,|u(y)|/\lambda)dy<+\infty$. Let $x\in\Omega$ be fixed and $\widetilde{M},\widetilde{M}_{x,\varepsilon}(s)$ be given by~\eqref{tildeM}. Using \ref{X1},  Lemma \ref{lem1}, \eqref{m<phim**} and Jensen's inequality we can write for all  $z\in B(x, {\varepsilon}/{2})$
\begin{equation*}
\begin{array}{lll} 
M(x,|u_\varepsilon(x)|/\lambda) &=&\frac{M(x,|u_\varepsilon(x)|/\lambda)}{M(z,|u_\varepsilon(x)|/\lambda)} \frac{M(z,|u_\varepsilon(x)|/\lambda)}{(\widetilde{M}_{x,\varepsilon})^{**}(|u_\varepsilon(x)|/\lambda)}
(\widetilde{M}_{x,\varepsilon})^{**}(|u_\varepsilon(x)|/\lambda)\\
&\leq& 4(\varphi(\varepsilon,u_\varepsilon/\lambda))^3 (\widetilde{M}_{x,\varepsilon})^{**}\Big(\frac{1}{\lambda}\int_\Omega J_\varepsilon(x-y)|u(y)|dy\Big)\\
&\leq& 4(\varphi(\varepsilon,c\varepsilon^{-N}/\lambda))^3\int_{|x-y|\leq \varepsilon} J_\varepsilon(x-y) (\widetilde{M}_{x,\varepsilon})^{**}(|u(y)|/\lambda)dy\\
&\leq& 4(\varphi(\varepsilon,c\varepsilon^{-N}/\lambda))^3\int_{|x-y|\leq \varepsilon} J_\varepsilon(x-y) \widetilde{M}_{x,\varepsilon}(|u(y)|/\lambda)dy\\
&\leq& 4(\varphi(\varepsilon,c\varepsilon^{-N}/\lambda))^3\int_{|x-y|\leq \varepsilon}J_\varepsilon(x-y)\widetilde{M}\big(y,|u(y)|/\lambda\big)dy.
\end{array}
\end{equation*}
Now we integrate both sides of the last inequality with respect to $x$ and  use Fubini's theorem, we obtain
\begin{equation*}
\begin{array}{lll}
\int_\Omega M(x,|u_\varepsilon(x)|/\lambda)dx&\leq& 4(\varphi(\varepsilon,c\varepsilon^{-N}/\lambda))^3 \int_\Omega \Big(\int_{|x-y|\leq \varepsilon}J_\varepsilon(x-y) \widetilde{M}(y,|u(y)|/\lambda)dy\Big)dx\\
&\leq& 4(\varphi(\varepsilon,c\varepsilon^{-N}/\lambda))^3 \int_\Omega \Big(\int_{\mathbb{R}^N}J_\varepsilon(x-y)dx\Big) \widetilde{M}(y,|u(y)|/\lambda)dy\\
&\leq& 4(\varphi(\varepsilon,c\varepsilon^{-N}/\lambda))^3 \int_\Omega M(y,|u(y)|/\lambda)dy,
\end{array}
\end{equation*}
This proves \eqref{eq1.8.9}. 

Applying the second claim of Lemma~\ref{lem1} instead of the first one and following the same lines we get the second assertion.\qed
 

\subsubsection*{Proof of Lemma~\ref{M(w_epsilon)<M(w)}}
Since $w\in L_M(\Omega)$. There exists $\lambda>0$ such that $\int_\Omega M(y,|w(y)|/\lambda)dy<+\infty$. Let $x\in\Omega$ be fixed.
Let $\widetilde{M}(y,s)=\lim_{\delta\to 0^+} \barint_{B(y,\delta)} M(z,s)\, dz\mbox{ and } \widetilde{M}_{x,\varepsilon_i}(s)=\inf_{y\in B(x,\varepsilon_i/2)} \widetilde{M}(y,s)$. Using \ref{X1},  \eqref{eq12'}, \eqref{m<phim**} and Jensen's inequality we can write for all  $z\in B(x,\varepsilon_i/2)$
\begin{equation*}
\begin{array}{lll}
M(x,|J_{\varepsilon_i}\ast(w)_{r_i}(x)|/\lambda) &=&\frac{M(x,|J_{\varepsilon_i}\ast(w)_{r_i}(x)|/\lambda)}{M(z,|J_{\varepsilon_i}\ast(w)_{r_i}(x)|/\lambda)} \frac{M(z,|J_{\varepsilon_i}\ast(w)_{r_i}(x)|/\lambda)}{(\widetilde{M}_{x,\varepsilon_i})^{**}(|J_{\varepsilon_i}\ast(w)_{r_i}(x)|/\lambda)}
(\widetilde{M}_{x,\varepsilon_i})^{**}(|J_{\varepsilon_i}\ast(w)_{r_i}(x)|/\lambda)\\
&\leq& 4(\varphi(\varepsilon_i,c\varepsilon_{i}^{-N} /\lambda))^3 (\widetilde{M}_{x,\varepsilon_i})^{**}\Big(\frac{1}{\lambda}\int_\Omega J_{\varepsilon_i}(x+r_{i}z_i-y)|w(y)|dy\Big)\\
&\leq& 4(\varphi(\varepsilon_i,c\varepsilon_{i}^{-N} /\lambda))^3\int_{|x+r_{i}z_i-y|\leq \varepsilon_i} J_{\varepsilon_i}(x-y) (\widetilde{M}_{x,\varepsilon_i})^{**}(|w(y)|/\lambda)dy\\
&\leq& 4(\varphi(\varepsilon_i,c\varepsilon_{i}^{-N} /\lambda))^3\int_{|x+r_{i}z_i-y|\leq \varepsilon_i} J_{\varepsilon_i}(x-y) \widetilde{M}_{x,\varepsilon_i}(|w(y)|/\lambda)dy\\
&\leq& 4(\varphi(\varepsilon_i,c\varepsilon_{i}^{-N} /\lambda))^3\int_{|x+r_{i}z_i-y|\leq \varepsilon_i}J_{\varepsilon_i}(x-y)\widetilde{M}\Big(y,|w(y)|/\lambda\Big)dy.
\end{array}
\end{equation*}
Integrating both sides of the last inequality with respect to $x$ and  using Fubini's theorem, we obtain
\begin{equation*}
\begin{array}{lll}
\int_\Omega M(x,|J_{\varepsilon_i}\ast(w)_{r_i}(x)|/\lambda)dx&\leq& 4(\varphi(\varepsilon_i,c\varepsilon_{i}^{-N} /\lambda))^3 \int_\Omega \Big(\int_{|x+r_{i}z_i-y|\leq \varepsilon_i}J_{\varepsilon_i}(x-y) \widetilde{M}(y,|w(y)|/\lambda)dy\Big)dx\\
&\leq& 4(\varphi(\varepsilon_i,c\varepsilon_{i}^{-N} /\lambda))^3 \int_\Omega \Big(\int_{|x+r_{i}z_i-y|\leq \varepsilon_i}J_{\varepsilon_i}(x-y)dx\Big) \widetilde{M}\big(y,|w(y)|/\lambda\big)dy\\
&\leq& 4(\varphi(\varepsilon_i,c\varepsilon_{i}^{-N} /\lambda))^3 \int_\Omega M(y,|w(y)|/\lambda)dy.
\end{array}
\end{equation*}
Applying the second claim of Lemma~\ref{lem3.4} instead of the first one and following the same lines we get the second assertion.\qed

\subsubsection*{Proof of Lemma~\ref{densityWM1}}
Let $u\in W^mL_M(\Omega)$ with $\supp u=\Omega^\prime$ and let $\varepsilon<\dist(\Omega^\prime,\partial \Omega)$. Then, $D^\alpha J_\varepsilon\ast u(x)=J_\varepsilon\ast D^\alpha u(x)$ in the distributional sense in $\Omega$. We fix arbitrary $\eta>0$. Since $D^\alpha u\in L_M(\Omega)$ for all $|\alpha|\leq m$, we can assume by Lemma~\ref{lem4.1} that there exists a sequence $(u_n^\alpha)$, where for each $n$ the function $u_n^\alpha$ is bounded and compactly supported in $\Omega$, and there exists $\lambda>0$ such that  
\begin{equation}\label{eq1.8.5} 
 \int_{\Omega} M(x, |D^\alpha u(x)- u_n^\alpha(x)|/ \lambda)dx\leq \eta.
\end{equation}
Observe that the  sequence $(u_n^\alpha)$ does not have to be in the derivative form. By the convexity of the $\Phi$-function $M$ with respect to its second argument, the Jensen inequality enables us to write
\begin{equation}\label{eq1.8.7}
\begin{array}{lll}
K&=& \int_{\Omega} M(x, |D^\alpha J_\varepsilon*u(x)-D^\alpha u(x)|/ 3\lambda)dx \\
&\leq&\frac{1}{3}\int_{\Omega} M(x, |J_\varepsilon* D^\alpha  u(x)-J_\varepsilon*u_n^\alpha(x)|/ \lambda)dx
+\frac{1}{3}\int_{\Omega}M(x,|J_\varepsilon * u_n^\alpha(x)-u_n^\alpha(x)|/ \lambda)dx\\
&+&\frac{1}{3}\int_{\Omega} M(x,|u_n^\alpha(x)-D^\alpha u(x)|/ \lambda)dx\\
&=& K_1+K_2+K_3.
\end{array}
\end{equation}
The term $K_3$ is already estimated by~\eqref{eq1.8.5}. To conclude the case of $K_1$, we apply Lemma \ref{lem3.4} and \eqref{eq1.8.5}. We get
\begin{equation}\label{eq1.8.6}
K_1 \leq 4(\varphi(2\varepsilon,c/\varepsilon^{N}))^3 K_3\leq 4(\varphi(2\varepsilon,c/\varepsilon^{N}))^3 \eta.
\end{equation}
As for the term $K_2$, by Jensen's inequality and Fubini's theorem we can write
\begin{equation*}\label{eq1.8.8}
\begin{array}{lll}
K_2 &\leq& \int_{\Omega}M\Big(x,|J_\varepsilon* u_n^\alpha(x)-u_n^\alpha(x)|/\lambda\Big)\\
&=& \int_{\Omega} M\Big(x, \int_{B(0,1)}\frac{J (y)}{\lambda} \big|u_n^\alpha(x-\varepsilon y)-u_n^\alpha(x)\big|dy\Big)dx\\
&\leq& \int_{B(0,1)} J(y)\int_{\Omega} M\Big(x, \frac{|u_n^\alpha(x-\varepsilon y)-u_n^\alpha(x)|}{\lambda}\Big)dxdy.
\end{array}
\end{equation*}
Then Lemma~\ref{lem3.1} applied to the function $u_n^\alpha$ yields that there exists $\varepsilon_{\eta}>0$ such that for every $|y|< 1$  and $\varepsilon\leq\varepsilon_{\eta}$, we have
$\|u_n^\alpha(\cdot-\varepsilon y)-u_n^\alpha(\cdot)\|_{L_{M}(\Omega)}\leq\lambda\eta$. Thus, for $\varepsilon<\varepsilon_{\eta}$ we get
$$
\int_{\Omega} M\Big(x, \frac{|u_n^\alpha(x-\varepsilon y)-u_n^\alpha(x)|}{\lambda}\Big)dx \leq \eta,
$$
which implies
\begin{equation}\label{eq1.9.1}
K_2  \leq \eta \int_{B(0,1)} J(y) dy=\eta.
\end{equation}
Putting all \eqref{eq1.8.5}, \eqref{eq1.8.6} and \eqref{eq1.9.1} together in \eqref{eq1.8.7}, we obtain
\begin{equation*}
I  \leq \eta\left[4(\varphi(2\varepsilon,c/\varepsilon^{N}))^3 +1\right].
\end{equation*}
Since $\eta$ is arbitrary, we get
\begin{equation}\label{Ktozero}
K= \int_{\Omega} M(x, |D^\alpha J_\varepsilon * u(x)-D^\alpha u(x)|/ 3\lambda)dx  \rightarrow 0\mbox{ as } \varepsilon\rightarrow 0.
\end{equation}
Therefore,
\begin{equation*}
\sum_{|\alpha|\leq m}\int_{\Omega} M(x, |D^\alpha J_\varepsilon * u(x)-D^\alpha u(x)|/3\lambda)dx\rightarrow 0\mbox{ as } \varepsilon\rightarrow 0.
\end{equation*}
\par Now, if $u\in W^mL_M(\Omega)\cap W^{m,p}(\Omega)$ we obtain the result in a similar way, taking into account the second claim of Lemma~\ref{lem3.4}.\qed

\subsubsection*{Proof of Conjecture~\ref{conj} in the case $m=1$}
Since $\supp\,u\subset\subset\Omega,$ we can take $\ve$ small enough so that $\Omega_\ve:=\supp\,(J_\ve \ast u)\subset\subset\Omega$. We fix $\eta$ and choose $\lambda>0$ and $\ve_\eta>0$, such that for every $\ve<\ve_\eta$ 
\begin{equation*} 
\int_{\Omega_{\ve_\eta}} M\Big(x,\frac{| D u(x)-D(J_\ve \ast u)(x)|}{\lambda}\Big)dx\leq \eta.
\end{equation*} by repeated arguments of \eqref{eq1.8.5}--\eqref{Ktozero} from the proof of Lemma~\ref{densityWM1} (in the case of $m=1$). This implies the convergence  $J_{\ve}\ast Du\xrightarrow[\ve\to0]{M}D u$ in $L_{M}(\Omega)$ for the gradients only. To get the strong convergence of the approximate sequence in $W^{1,1}(\Omega)$, we observe first that $|{\Omega_{\ve_\eta}}|<\infty$ and by the Jensen inequality we can write
$$
M\Big(x,\frac{1}{|{\Omega_{\ve_\eta}}|} \int_{\Omega_{\ve_\eta}}  \frac{| D u(x)-D(J_\ve \ast u)(x)|}{\lambda} dx\Big)\leq\frac{1}{|{\Omega_{\ve_\eta}}|}
\int_{{\Omega_{\ve_\eta}}}M\Big(x,\frac{|D u(x)-D(J_\ve \ast u)(x)|}{\lambda}\Big)dx.
$$
Since $M(x,t)=0$ if and only if $t=0$, the last inequality yields 
$$
\lim_{\ve\to0^+}\int_{\Omega_{\ve_\eta}}| D u(x)-D(J_\ve \ast u)(x)| dx=0.
$$
We then conclude the conjecture by using the Poincar\'{e} inequality 
\[
\int_{\Omega_{\ve_\eta}} \Big| u(x)-(J_\ve \ast u)(x)\Big| dx\leq C \int_{\Omega_{\ve_\eta}}   \Big| D u(x)-D(J_\ve \ast u)(x)\Big|dx.\]  
\qed
 
\subsubsection*{Proof of Corollary~\ref{coro:d-p}}
The proof of the modular convergence follows the same lines as the~proof of~Conjecture~\ref{conj}, when we assume~\eqref{M>p}.  Note that due to the growth condition and the Poincar\'e inequality we get the strong convergence of the approximate sequence in $W^{1,p}(\Omega)$. Indeed, we have\[
C\int_\Omega |u|^p\,dx\leq c\int_\Omega |\nabla u/\lambda|^p\,dx\leq \int_\Omega M(x,|\nabla u|/\lambda)\,dx.\]  
\qed

\section*{Acknowledgements} This paper was discussed when A.Y. was spending  a scientific stay at  IMPAN (Institute of Mathematics of~Polish Academy of Sciences, Warsaw) in July 2017. A.Y. would like to thank warmly A. \'Swierczewska-Gwiazda and P.~Gwiazda for their kind invitation, their welcome and hospitality. The visit of A.Y. was funded by Warsaw Center of Mathematics and Computer Science.

The research of P.G. and I.C. has been supported by the NCN grant  no. 2014/13/B/ST1/03094. The work was also partially supported by the Simons Foundation grant 346300 and the Polish Government MNiSW 2015--2019 matching fund.

\section*{References} 
\bibliographystyle{plain}
\bibliography{yags}

\end{document}